 \documentclass[accepted]{uai2021} % after acceptance, for a revised
\usepackage[american]{babel}
\usepackage{natbib} % has a nice set of citation styles and commands
\usepackage{tikz} % nice language for creating drawings and diagrams

\usepackage{microtype}
\usepackage{graphicx}
\usepackage{subfigure}
\usepackage{booktabs} % for professional tables
\usepackage{amsmath,amssymb,array,verbatim, amsthm, dsfont,mathrsfs}
\usepackage{mathtools}

\usepackage{hyperref}

\DeclareMathOperator{\ee}{\mathbb{E}}			% expected value
\DeclareMathOperator{\prob}{{P}}			% probability
\newcounter{l1}
\newcounter{l2}
\newcounter{l3}
\newcommand{\bdotlist}{\begin{list}{$\bullet$}{}}
\newcommand{\bboxlist}{\begin{list}{$\Box$}{}}
\newcommand{\bbboxlist}{\begin{list}{\raisebox{.005in}{{\tiny $\blacksquare$ \ \ }}}{}}
\newcommand{\bdashlist}{\begin{list}{$-$}{} }
\newcommand{\blist}{\begin{list}{}{} }
\newcommand{\barablist}{\begin{list}{\arabic{l1}}{\usecounter{l1}}}
\newcommand{\balphlist}{\begin{list}{(\alph{l2})}{\usecounter{l2}}}
\newcommand{\bAlphlist}{\begin{list}{\Alph{l2}.}{\usecounter{l2}}}
\newcommand{\bdiamlist}{\begin{list}{$\diamond$}{}}
\newcommand{\bromalist}{\begin{list}{(\roman{l3})}{\usecounter{l3}}}

\newtheorem{remark}{Remark}

\newtheorem{theorem}{Theorem}
\newtheorem{lemma}{Lemma}

\newtheorem{proposition}{Proposition}

\title{Optimal communication and control strategies in a multi-agent  MDP problem}

\author[1]{\href{mailto:Sagar Sudhakara <sagarsud@usc.edu>?Subject=Your UAI 2021 paper}{Sagar Sudhakara}{}} % Lead author
\author[1]{Dhruva Kartik}
\author[1]{Rahul Jain}
\author[1]{Ashutosh Nayyar}

\affil[1]{%
    Electrical Engineering  Dept.\\
    University of Southern California \\
    Los Angeles, California, USA
}

\begin{document}
\maketitle

%\icmltitle{Optimal communication and control strategies in a multi-agent  MDP problem}

\begin{abstract}
The problem of controlling multi-agent systems under different models of information sharing among agents has received significant attention in the recent literature. In this paper, we consider a setup where rather than committing to a fixed information sharing protocol (e.g.  periodic sharing or no sharing  etc), agents can dynamically decide at each time step whether to share information with each other and incur the resulting communication cost. This setup requires a joint design of agents' communication and control strategies in order to optimize the trade-off between communication costs and control objective. We first show that agents can ignore a big part of their private information without compromising the system performance. We then  provide a common information approach based solution for the strategy optimization problem. This approach relies on constructing a fictitious POMDP whose solution (obtained via a dynamic program) characterizes the optimal strategies for the agents. We also show that our solution can be easily modified to incorporate constraints on when and how frequently agents can communicate.
\end{abstract}
%%Finally, we exploit certain features of our problem to show that the POMDP dynamic program can be reformulated as a MDP dynamic program with an appropriately defined state.

\section{Introduction}\label{introduction}
%\begin{abstract}
%The problem of controlling multi-agent systems under different models of information sharing among agents has received significant attention in the recent literature. In this paper, we consider a setup where rather than committing to a fixed information sharing protocol (e.g.  periodic sharing or no sharing  etc), agents can dynamically decide at each time step whether to share information with each other and incur the resulting communication cost. This setup requires a joint design of agents' communication and control strategies in order to optimize the trade-off between communication costs and control objective. We first show that agents can ignore a big part of their private information without compromising the system performance. We then  provide a common information approach based solution for the strategy optimization problem. This approach relies on constructing a fictitious POMDP whose solution (obtained via a dynamic program) characterizes the optimal strategies for the agents. We also show that our solution can be easily modified to incorporate constraints on when and how frequently agents can communicate.
%\end{abstract}

 The problem of sequential decision-making by a team of collaborative agents has received significant attention in the recent literature. The goal in such problems is to jointly design decision/control strategies for the multiple agents in order to optimize a performance metric for the team. The nature of this joint strategy optimization problem as well as the best achievable performance depend crucially on the \emph{information structure} of the problem.  Intuitively, the information structure of a multi-agent problem specifies what information is available to each agent at each time. Depending on the underlying communication environment, a wide range of information structures can arise. If communication is costless and unrestricted, all agents can share all information with each other.  If communication is too costly or physically impossible, agents may not be able to share any information at all.    It could also be the case that agents can communicate only periodically or that  the ability to communicate varies among the agents leading to one-directional communication between certain pairs of agents. Each of these communication models corresponds to a different information structure which, in turn, specifies the class of feasible decision/control strategies for the agents.

In this paper, we consider a setup where rather than committing to a fixed information sharing protocol (e.g.  periodic sharing or no sharing  etc), agents can dynamically decide at each time step whether to share information with each other and incur the resulting communication cost. Thus, at each time step, agents have to make two kinds of decisions - communication decisions  that govern the information sharing as well as control decisions that govern the evolution of the agents' states. The two kinds of agents' strategies - communication strategies  and control strategies - need to be jointly designed in order to optimize the trade-off between communication costs and control objective.  

\textbf{Contributions} (i) We first show that agents can ignore a big part of their private information without compromising the system performance. This is done by using an agent-by-agent argument where we fix the strategies of one agent arbitrarily and  find a sufficient statistic for the other agent.  This sufficient statistic turns out be a subset of the agent's private information. This reduction in private information narrows down the search for optimal strategies to a class of simpler strategies. (ii) We then  adopt the common information based solution approach for finding the optimal strategies. This approach relies on constructing an equivalent POMDP from the perspective of a fictitious coordinator that knows the common information among the agents. The solution of this POMDP (obtained via a dynamic program) characterizes the optimal strategies for the agents. (iii) Finally, we extend our setup to incorporate several constraints on when  and  how frequently agents can communicate with each other.  We show that our solution approach can be easily modified to incorporate these constraints using a natural augmentation of the state in the coordinator's POMDP.

\textbf{Related Work} There is a significant body of prior work on decentralized control and decision-making in multi-agent systems. We focus on works where the dynamic system can be viewed as a  Markov chain jointly being controlled by multiple agents/controllers.  We can organize this literature based on the underlying information structure (or the information sharing protocol). 

In Dec-MDPs and Dec-POMDPS, each agent receives a partial or noisy observation of the current system state \citep{bernstein2002complexity}. These agents cannot communicate or share their observations with each other and can only use their private action-observation history to select their control actions.
%Dec-POMDPs capture a wide class of decentralized decision-making models \citep{foerster2016learning,xie2020optimally} including ours. 
Several methods for solving such generic Dec-POMDPs exist in the literature \citep{szer2012maa,seuken2008formal,kumar2015probabilistic,dibangoye2016optimally,rashid2018qmix,hu2019simplified}. However, these generic methods either involve prohibitively large amount of computation or cannot guarantee optimality. For certain Dec-MDPs and Dec-POMDPs with additional structure, such as transition independence in factored Dec-MDPs \citep{becker2004solving} or one-sided information sharing \citep{xie2020optimally}, one can derive additional structural properties of the optimal strategy and use these properties to make the computation more tractable.

In decentralized stochastic control literature, a variety of information structures (obtained from different information sharing protocols) have been considered \citep{nayyar2010optimal, nayyar2013decentralized, mahajan2013optimal}. For example, \citep{nayyar2010optimal} considers the case where agents share their information with each other with a fixed delay; \citep{nayyar2013decentralized} provides a unified treatment for a range of information sharing protocols including periodic sharing, sharing of only control actions etc.  \citep{mahajan2013optimal,foerster2019bayesian} consider a setup where only the agents' actions are shared with others. 

In emergent communication, agents have access to a cheap talk channel which can be used for communication. \citep{sukhbaatar2016learning,foerster2016learning,cao2018emergent} propose methods for jointly learning the control and communication strategy in such settings. The key communication issue in these works is to design the most effective way of encoding the available information into the communication alphabet. In contrast, the communication issue in our setup is whether the cost of sharing states is worth the potential control benefit.

In this paper, we consider a model where the information sharing protocol is not fixed a priori. At each time, agents in our model make an explicit choice regarding sharing their information with each other. We seek to jointly design this information sharing strategy and the agents' control strategies. 
This problem and many of the problems considered in the prior literature can be reduced to Dec-POMDPs by a suitable redefinition of states, observations and actions. However, as demonstrated in \citep{xie2020optimally}, a generic Dec-POMDP based approach for problems with (limited) inter-agent communication involves very large amount of computation since it ignores the underlying communication structure. Instead, we derive some structural properties of the strategies that significantly simplify the strategy design. We then provide a dynamic program based solution using the common information approach. To the best of our knowledge, our information sharing mechanism has not been analyzed before.

\textbf{Notation}
Random variables are denoted with upper case letters( $X$, $Y$, etc.), their realization with lower case letters ($x$, $y$, etc.), and their space of realizations by script letters ($\mathcal{X}$, $\mathcal{Y}$, etc.). Subscripts denote time and superscripts denote the subsystem; e.g., $X^i_t$ denotes the state of subsystem $i$ at time $t$. The short hand notation $X^i_{1:t}$ denotes the vector $(X^i_1,X^i_2,...,X^i_t)$. 
The face letters denote the collection of variables at all subsystems; e.g. ${X}_t$, denotes $(X^1_t,X^2_t)$. $\triangle(\mathcal{X})$ denotes the probability simplex for the space $\mathcal{X}$. $\prob(A)$ denotes the probability of an event $A$. $\ee[X]$ denotes the expectation of a random variable $X$. $\mathds{1}[x=y]$ denotes the indicator function of the statement $x=y$, i.e. $\mathds{1}[x=y]$, is $1$ if $x=y$ and $0$ otherwise. For simplicity of notation, we use $\prob(x_{1:t},u_{1:t-1})$ to denote $\prob(X_{1:t}=x_{1:t},U_{1:t-1}=u_{1:t-1})$ and a similar notation for conditional probability. We use $-i$ to denote agent/agents other than agent $i$.  

\section{Problem Formulation}\label{sec:problem_formulation}
%\red{bold symbols have been used inconsistently to denote pairs; I suggest just not using bold}

 Consider a discrete-time  control system with two agents. Let $X^i_t \in \mathcal{X}^i$ denote the local state of agent $i$. $X_t:=(X^1_t,X^2_t)$ represents the local state of both agents. The initial local states of all agents are independent random variables with initial local state $X^i_1$ having the probability distribution $P_{X^i_1}$, $i=1,2$. Each agent perfectly observes its own local state.  Let $U^i_t \in \mathcal{U}^i$ denote the control action of agent $i$ at time $t$ and $U_t:=(U^1_t,U^2_t)$ denote the collection of all control actions at time $t$. 
 The local state of agent $i, i=1,2$, evolves according to 
\begin{equation}
    X^i_{t+1}=k^i_t(X^{i}_{t},U^i_t,W^i_t) \label{dyna}
\end{equation}
where $W^i_t \in \mathcal{W}^i$ $i, i=1,2$ is the disturbance in dynamics with  probability distribution $P_{W^i}$. The initial state $X_1$ and the disturbances  $\{W^i_t\}_{t=1}^{\infty}$, $i=1,2$, are independent random variables.
Note that the next local state of agent $i$ depends on the current local state and control action of agent $i$.

In addition to deciding the control actions at each time, the two agents need to decide whether or not to initiate communication at each time. We use the binary variable $M^i_t$ to denote  the communication decision taken by agent $i$. If either agent initiates communication (i.e., if $M^i_t=1$ for any $i$), then the agents share their local state information with each other.
% In other words,  if we def $M^{or}_t:=(M^1_t\vee M^2_t)$. If $M^{or}_t=1$, then local states are communicated between the agents and for $M^{or}_t=0$ local states are not communicated between agents.
% We consider the full observation model, control station $i$ perfectly observes the local state $X^i_{1:t}$. In addition to the observation of the state of its agent, each control station perfectly observes the one-step delayed control actions $U_{1:t-1}$ and the local states of other agent if $M^{or}_t=1$ i.e. communication of state variables takes place between agents when $M^{or}_t=1$. Let's define variable $Z_t$ as:
Let $M^{or}_t:=\max(M^1_t, M^2_t)$ and let $Z_t$ represent the information exchanged between the agents at time $t$. Then, based on the communication model described above, we can write
\begin{equation}\label{zt}
  Z_t=\begin{cases}
    X^{1,2}_t, & \text{if $M^{or}_t=1$}.\\
    \phi, & \text{if $M^{or}_t=0$}.
  \end{cases}
\end{equation}
\textbf{Information structure and decision strategies:}
At the beginning of the $t$-th time step, the information available to agent $i$  is given by 
\begin{equation}
    I^i_{t}=\{X^i_{1:t},U^i_{1:t-1},Z_{1:t-1},M^{1,2}_{1:t-1}\}.
\end{equation}
Agent $i$ can use this information to make its communication decision at time $t$. Thus, $M^i_t$ is chosen as a function of $I^i_t$ according to
%The control stations perfectly recall all the data they observe. Thus, in the full observation model, control station $i$ chooses a communication action according to 
\begin{equation}\label{commact}
    M^i_{t}=f^i_t(I^i_{t}),
\end{equation}
where the function $f^i_t$ is referred to as the communication strategy of agent $i$ at time $t$.
After the communication decisions are made and the resulting communication (if any) takes place, the information available to  agent $i$ is 
\begin{equation}
    I^i_{t^+}=\{I^i_{t},Z_t,M^{1,2}_t\}.
\end{equation} 
Agent $i$ then chooses its control action according to 
\begin{equation}\label{controlact}
    U^i_{t}=g^i_t(I^i_{t^+}),
\end{equation}
where the function $g^i_t$ is referred to as the control strategy of agent $i$ at time $t$.
%The function $f^i_t$, $g^i_t$ are called the communication law and control law of control station $i$ respectively.

 $f^i:=(f^i_1,f^i_2,...,f^i_T)$ and  $g^i:=(g^i_1,g^i_2,...,g^i_T)$
are called the communication and control strategy of agent $i$ respectively.
% The collection $f:=(f^1,f^2)$, $g:=(g^1,g^2)$ of communication and control strategies of all agents are called the communication and control strategy of the system.

\textbf{Strategy optimization problem:} At time $t$, the system incurs a cost $c_t(X^1_t,X^2_t,U^1_t,U^2_t)$ that depends on the local states and control stations of both agents. Thus, the agents are coupled through cost. In addition, a communication cost $\rho$ is incurred each time the agents share local states with each other.
The system runs for a time horizon $T$. The objective is to find communication and control strategies for the two agents in order to minimize the expected value of the  sum of control and communication costs over the time horizon $T$:
\begin{equation}\label{eq:cost}
   \ee\Big[\sum_{t=1}^{T}c_t({X}_t,{U}_t)+\rho\mathds{1}_{\{M^{or}_t=1\}}\Big].
\end{equation}
%where the expectation is with respect to a joint measure of $(\textbf{X}_{1:T},\textbf{U}_{1:T},\textbf{M}_{1:T})$ induced by the choice of the communication and control strategy $f$, $g$. We are interested in the following optimal control problem.

%\textbf{Problem 1}: Give the distributions $P_{X^i}$, $P_{W^i}$ of the initial local state and plant disturbance of agent $i$, $i=1,2$, a horizon $T$, the cost function $c_t$ and communication cost $\rho$, $t=1,2,...,T$, find a control strategy $g$ and communication strategy $f$ that minimizes the expected total cost given by (6).
%The above model and optimization problem arise in a variety of Reinforcement learning and communication applications.

\section{Preliminary Results and Simplified Strategies}\label{Preliminary_Results}
In this section we show that agents can ignore parts of their information without losing optimality. This removal of information narrows the search for optimal strategies to a class of simpler strategies and is a key step in our approach for finding optimal strategies. To proceed, we first split the information available to the agents into two parts -- common information (which is available to both agents) and private information (which is everything except the common information):
\begin{enumerate}
\item At the beginning of time step $t$, before the communication decisions are made, the common information is defined as 
\begin{align}
C_t:=(Z_{1:t-1},M^{1,2}_{1:t-1}).\label{commoninfo}
\end{align}
\item After the communication decisions are made and the resulting communication (if any) takes place, the common information is defined as 
\begin{align}
C_{t^+}=(Z_{1:t},M^{1,2}_{1:t}).\label{commoninfoplus}
\end{align}
\end{enumerate}
The following lemma establishes a key conditional independence property that will be critical for our analysis.
% the independence proposition stated below and person-by-person approach to show that the past values of the local state $X^i_{1:t-1}$ and control actions $U^i_{1:t-1}$ are irrelevant at control station for agent $i$ at time $t$. To pursue this we split the data at each control station into two parts: the common data(or shared data) $C_t=(Z_{1:t-1},M_{1:t-1})$ (before communication) and $C_{t^+}=(Z_{1:t},M_{1:t})$ (after communication) which is observed by all control stations.
\begin{lemma}[Conditional independence property] \label{LEM:INDEPEN}
Consider any arbitrary choice of agents' strategies. Then, at any time $t$, the two agents'  local states  and control actions   are conditionally independent given the common information $C_t$ (before communication) or $C_{t^+}$ (after communication). That is, if $c_t, c_{t^+}$ are the realizations of the common information before and after communication respectively, then for any realization ${x}_{1:t}, {u}_{1:t-1}$ of states and actions, we have
\begin{equation}\label{eq:indepen}
    \prob({x}_{1:t},{u}_{1:t-1}|c_t)=\displaystyle\prod_{i=1}^{2} \prob(x^{i}_{1:t},{u}^{i}_{1:t-1}|c_t), 
    \end{equation}
    \begin{equation}\label{eq:indepen2}
      \prob({x}_{1:t},{u}_{1:t}|c_{t^+})=\displaystyle\prod_{i=1}^{2} \prob(x^{i}_{1:t},{u}^{i}_{1:t}|c_{t^+}).  
    \end{equation}
    Further, $\prob(x^{i}_{1:t},{u}^{i}_{1:t-1}|c_t)$ depends on only on agent $i$' strategy and not on the strategy of agent $-i$.
\end{lemma}    
\begin{proof}
See Appendix \ref{proof:CI}.
\end{proof}

The following proposition shows that agent $i$ at time $t$ can ignore its past states and actions, i.e. $X^i_{1:t-1}$ and $U^i_{1:t-1}$, without losing optimality. This allows agents to use simpler strategies where the communication and control decisions are functions only of the current state and the common information.
\begin{proposition} \label{PROP:ONE}
 Agent $i$, $i=1,2,$ can restrict itself to strategies of the form below
\begin{equation} \label{eq:structure1}
   M^i_t= \bar{f}^i_t(X^i_{t},C_t)
\end{equation}
%M^i_t= \bar{f}^i_t(X^i_{t},Z_{1:t-1},M^{1,2}_{1:t-1})
% U^i_t= \bar{g}^i_t(X^i_{t},Z_{1:t},M^{1,2}_{1:t})
\begin{equation}  \label{eq:structure2}
   U^i_t= \bar{g}^i_t(X^i_{t},C_{t^+})
\end{equation}
without loss of optimality.  In other words, at time $t$, agent $i$ does not need the past local states and actions, $X^i_{1:t-1},U^i_{t-1}$, for making optimal decisions.
\end{proposition}
\begin{proof}
To  prove this result, we fix agent $-i$'s strategy to an arbitrary choice and then show that agent $i$'s  decision problem can be modeled as an MDP in a suitable state space. The result then follows from the fact that Markovian strategies are optimal in an MDP.  See Appendix \ref{proof:prop1} for details.
\end{proof}
%Thus, the past values of local state and control actions are irrelevant at control
%station for agent $i$ at time $t$. However, even after shedding $X^i_{1:t-1}$ and $U^i_{1:t-1}$ the data at each control station is still increasing with time.

\section{Centralized Reformulation Using Common Information}\label{Coordinator Results}
In this section, we provide a centralized reformulation of the multi-agent strategy optimization problem using the common information approach of \citep{nayyar2013decentralized}.
 The main idea of the proof is to formulate an equivalent single-agent POMDP problem; solve the equivalent POMDP using a dynamic program; and translate the results back to the original problem.
 
Because of Proposition \ref{PROP:ONE}, we will only consider strategies of the form given in \eqref{eq:structure1} and \eqref{eq:structure2}. 
 %Note that the common information is increasing with time (i.e., $C_t\subset C_{t+1}$ ), while the local information $X^i_t$ for all agents has a fixed size. Thus, the system has partial history sharing information structure with finite local memory. Nayyar.et al. [] derived structural properties of optimal controllers and a dynamic programming decomposition for such an information structure.\\
%To present the result, we proceed as follows:\\
Following the approach in \citep{nayyar2013decentralized},  we construct an equivalent problem   by adopting  the point of view of a fictitious coordinator that observes only the common information among the agents ( i.e., the coordinator observes  $C_t$ before communication and $C_{t^+}$  after $Z_t$ happens), but not the current local state (i.e., $X^i_t,i=1,2$).
Before communication at time $t$, the coordinator chooses a pair of  \emph{prescriptions}, $\Gamma_t:=(\Gamma^1_t,\Gamma^2_t)$,
where $\Gamma^i_t$ is a mapping from $X^i_t$ to $M^i_t$ (more precisely,  $\Gamma^i_t$ maps $\mathcal{X}^i$ to $\{0,1\}$). The interpretation of the prescription is that it is a directive for the agents about how they should use their local state information to make the communication decisions. Thus, agent $i$ generates its communication decision by evaluating the function $\Gamma^i_t$ on its current local state: 
\begin{align}
    &M^i_t=\Gamma^i_t(X^i_t).\label{commact1}
    \end{align}   
Similarly, after the communication decisions are made and $Z_t$ is realized, the coordinator chooses a     pair of  \emph{prescriptions}, $\Lambda_t:=(\Lambda^1_t,\Lambda^2_t)$,
where $\Lambda^i_t$ is a mapping from $X^i_t$ to $U^i_t$ (more precisely,  $\Lambda^i_t$ maps $\mathcal{X}^i$ to $\mathcal{U}^i$). Agent $i$ then generates its control action by evaluating the function $\Lambda^i_t$ on its current local state: 
\begin{align}
    &U^i_t=\Lambda^i_t(X^i_t).\label{controlact1}
    \end{align}   
The coordinator chooses its prescriptions based on the common information. Thus, 
  \begin{align}
    &\Gamma^1_t=d^1_t(C_t), ~~ \Gamma^2_t=d^2_t(C_t),\notag\\
    &\Lambda^1_t=d^1_{t^+}(C_{t^+}), ~~ \Lambda^2_t=d^2_{t^+}(C_{t^+}),
\end{align}
where $d^1_t, d^2_t, d^1_{t^+}, d^2_{t^+}$ are referred to as \emph{coordinator's communication and control strategy for the two agents at time $t$}. The collection of functions $(d^1_1, d^2_1,d^1_{1^+},...,d^1_{T^+},d^2_{T^+})$ is called the coordinator's strategy.  The coordinator's strategy optimization problem is to find a coordination strategy to minimize the expected total cost given by \eqref{eq:cost}.
The following Lemma shows the equivalence of the coordinator's strategy optimization problem and  the original strategy optimization problem for the agents.

\begin{lemma}\label{lem:equi}
Suppose that $(d^{1*}_1, d^{2*}_1,...,d^{1*}_{T^+},d^{2*}_{T^+})$ is the optimal strategy for the coordinator. Then, optimal communication and control strategies for the agents in the original problem can be obtained as follows: for $i=1,2$,
\begin{align}
&\bar{f}^{i*}_t(X^i_{t},C_t) = \Gamma^i_t(X^i_t) ~~ \mbox{where~}  \Gamma^i_t=d^{i*}_t(C_t), \\
&\bar{g}^{i*}_t(X^i_{t},C_t) = \Lambda^i_t(X^i_t) ~~ \mbox{where~}  \Lambda^i_t=d^{i*}_{t^+}(C_t).
\end{align}
\end{lemma}
\begin{proof}
The lemma is a direct consequence of the results in \citep{nayyar2013decentralized}.
\end{proof}

Lemma \ref{lem:equi} implies that the agents' strategy optimization problem can be solved by solving the coordinator's strategy optimization problem. The advantage of the coordinator's problem is that it is a sequential decision-making problem with coordinator as the only decision-maker. (Note that once the coordinator makes its decisions about which prescription to use, the agents act as mere evaluators and not independent decision-makers.) 

%\blue{Beliefs: Independence and belief update}

\textbf{Coordinator's belief state:} As shown in \citep{nayyar2013decentralized}, the coordinator's problem can be viewed as a POMDP. Therefore, the coordinator's belief state will serve as the sufficient statistic for selecting prescriptions. Before communication at time $t$, the coordinator's belief is given as:
\begin{align}
    &\Pi_t(x^1,x^2)=P(X^1_t=x^1, X^2_t=x^2|C_t,\Gamma_{1:(t-1)},\Lambda_{1:(t-1)}).
\end{align}
After the communication decisions are made and $Z_t$ is realized, the coordinator's belief is given as:
\begin{align}
    &\Pi_{t^+}(x^1,x^2)=P(X^1_t=x^1, X^2_t=x^2|C_{t^+},\Gamma_{1:t},\Lambda_{1:(t-1)}).
\end{align}
Because of conditional independence property identified in Lemma \ref{LEM:INDEPEN}, the coordinator's beliefs  can be factorized into beliefs on each agent's state, i.e.,
\begin{equation}
    \Pi_t(x^1,x^2)=\Pi^1_t(x^1) \Pi^2_t(x^2),
\end{equation}
\begin{equation}
    \Pi_{t^+}(x^1,x^2)=\Pi^1_{t^+}(x^1) \Pi^2_{t^+}(x^2),
\end{equation}
where $\Pi^i_t$ (resp. $\Pi^i_{t^+}$ ) is the coordinator's belief  on $X^i_t$ before (resp. after) the communication decisions are made. The coordinator can update its beliefs on the agents' states in a sequential manner as described in the following lemma.
\begin{lemma}\label{LEM:UPDATE}
$\Pi^i_1$ is the prior belief ($P_{X^i_1}$) on the initial state  $X^i_1$  and for each $t \geq 1$, 
\begin{equation}
    \Pi^i_{t^+}=\eta^i_t( \Pi^i_{t},\Gamma^i_t,Z_t),
\end{equation}
\begin{equation}
    \Pi^i_{t+1}=\beta^i_t( \Pi^i_{t^+},\Lambda^i_{t}),
\end{equation}
where   $\eta^i_t, \beta^i_t$ are fixed functions derived from the system model. (We will use $\beta_t( \Pi^{1,2}_{t^+},\Lambda^{1,2}_{t})$ to denote the pair $\beta^1_t( \Pi^1_{t^+},\Lambda^1_{t}), \beta^2_t( \Pi^2_{t^+},\Lambda^2_{t})$. Similar notation will be used for the pair $\eta^1_t(\cdot), \eta^2_t(\cdot)$.)
%\red{Shouldn't these by $\beta^i_t$ and $\eta^i_t$? Is the $\beta_t$ notation in DP consistent with this lemma?}
\end{lemma}
\begin{proof}
See Appendix \ref{proof:Coord}.
\end{proof}

Finally, we note that if the coordinator's beliefs at time $t$ (before communication) are $\Pi^1_t, \Pi^2_t$ and it selects the prescriptions $\Gamma^1_t, \Gamma^2_t,$ the probability that $Z_t=\phi$ (i.e. probability of no communication at time $t$) is given as 
\begin{align}
&P(Z_t=\phi|\Pi_t^{1,2},\Gamma_t^{1,2}) = \notag\\
   &~~ =\sum_{x^1, x^2}\mathds{1}_{\{\Gamma^1_t(x^{1})=0\}}\mathds{1}_{\{\Gamma^2_t(x^{2})=0\}}\Pi^1_t(x^1)\Pi^2_t(x^2). \label{eq:probz1}
\end{align}
Similarly, the probability that $Z_t = (x^1,x^2)$ is given as 
\begin{align}
    &P(Z_t=(x^1,x^2)|\Pi_t^{1,2},\Gamma_t^{1,2})\notag\\
    &~~=\Big[\max(\Gamma^1_t(x^{1}),\Gamma^2_t(x^{2}))\Big]\Pi^1_t(x^1)\Pi^2_t(x^2). \label{eq:probz2}
 \end{align}
 (Note that $\max(\Gamma^1_t(x^{1}),\Gamma^2_t(x^{2}))$ is equal to the indicator function of the union of the events 
 $\{\Gamma^1_t(x^{1}) \neq 0\}$ and $\{\Gamma^2_t(x^{2}) \neq 0\}$.)

%\blue{POMDP dp with explicit summations and updates}

\textbf{Coordinator's dynamic program:} Using Lemma 3 and the probabilities given in \eqref{eq:probz1} - \eqref{eq:probz2}, we can write a dynamic program for the coordinator's POMDP problem. In the following theorem, $\pi^i$ denotes a general probability distribution on $\mathcal{X}^i$ and $\delta_{x^i}$ denotes a delta distribution centered at $x^i$.

\begin{theorem}\label{thm:dp}
The value functions  for the coordinator's dynamic program as follows: $V_{T+1}(\pi^1,\pi^2):=0$ for all $\pi^1,\pi^2,$ and for  $t=T, \ldots, 2, 1,$ \\
 \begin{align}
  &V_{t^+}(\pi^1,\pi^2) := \notag \\
 & \min_{\lambda^1, \lambda^2} \Big[\Big(\sum_{x^{1,2}}c_t\big(x^{1,2},\lambda^1(x^{1}),\lambda^2(x^{2})\big)\pi^1(x^1)\pi^2(x^2)\Big) \notag\\
 &+V_{t+1}(\beta_t( \pi^{1,2},\lambda^{1,2}))\Big],
\end{align}
where $\beta_t$ is as described in Lemma \ref{LEM:UPDATE}.
  \begin{align}
  &V_{t}(\pi^1,\pi^2) := \notag \\
 &\min_{\gamma^1, \gamma^2}\Big[\rho\sum_{x^{1,2}}\max(\gamma^1(x^{1}),\gamma^2(x^{2}))\pi^1(x^1)\pi^2(x^2) \notag \\ 
&+P(Z_t=\phi|\pi^{1,2},\gamma^{1,2}) V_{t^+}(\eta_t( \pi^{1,2},\gamma^{1,2}, \phi))\notag\\ &+\sum_{\tilde{x}^{1,2}}P(Z_t=\tilde{x}^{1,2}|\pi^{1,2},\gamma^{1,2})V_{t^+}(\delta_{\tilde{x}^1},\delta_{\tilde{x}^2})\Big], \label{dyn:update}
\end{align}
where $\eta_t$ is as described in Lemma \ref{LEM:UPDATE} and $P(Z_t=\phi|\pi^{1,2},\gamma^{1,2})$, $P(Z_t=\tilde{x}^{1,2}|\pi^{1,2},\gamma^{1,2})$ are as described by \eqref{eq:probz1} - \eqref{eq:probz2}. The coordinator's optimal strategy is to pick the minimizing prescription pairs for each time and each $(\pi^1,\pi^2)$.
\end{theorem}
\begin{proof}
Since the coordinator's problem is a POMDP, it has a corresponding dynamic program. The value functions in the theorem can be obtained by simple manipulations of the POMDP dynamic program for the coordinator. 
%See Appendix \red{???}
\end{proof}

\begin{remark}\label{infiniteremark}
We can consider an infinite-horizon discounted cost analog of the problem formulation in this paper wherein the transition and cost functions are time-invariant. All the results can be extended to this discounted setting in a straightforward manner using the approach in \citep{nayyar2013decentralized}.
\end{remark}

%\red{Please write value functions for t+, t in a form similar to the above functions. Please pay close attention to the latex code}

%Value function $V_{T}(\pi_{T})$ before transmission at time $T$:
%\begin{equation}
%\begin{aligned}
%\min_{\gamma_T}&[\rho*\sum_{x_T^{1,2}}max(\gamma^1(x_T^{1}),\gamma^2(x_T^{2}))\pi^1(x^1_T)\pi^2(x^2_T)\\ 
%&+P(Z_t=\phi|\pi_T^{1,2},\Gamma_T^{1,2}) V_{T^+}(\eta( \Pi^{1,2}_{T},\Gamma^{1,2}_T))\\ &+\sum_{\tilde{x}_T^{1,2}}P(Z_t=x_T^{1,2}|\pi_T^{1,2},\Gamma_T^{1,2})V_{T^+}(\tilde{x}^1_T,\tilde{x}^2_T)]
%\end{aligned}
%\end{equation}
%  
%and for $t=T-1,T-2,...,1,$, value function $V_{t^+}(\pi_{t^+})$ after transmission:
%\begin{align}
%  \min_{\lambda_t}&[\sum_{x_t^{1,2}}c_T(x_t^{1,2},\lambda^1(x_t^{1}),\lambda^2(x_t^{2}))\pi^1_{t^+}(x^1_t)\pi^2_{t^+}(x^2_t)\notag\\
%  &+V_{t+1}(\beta(\pi_{t^+},\lambda_t))]
%\end{align}
%Value function $V_{t}(\pi_{t})$ before transmission:
%\begin{align}
%\min_{\gamma_t}&[\rho*\sum_{x_t^{1,2}}max(\gamma^1(x_t^{1}),\gamma^2(x_t^{2}))\pi^1(x^1_t)\pi^2(x^2_t)\notag\\ 
%&+P(Z_t=\phi|\pi_t^{1,2},\Gamma_t^{1,2}) V_{t^+}(\eta( \Pi^{1,2}_{t},\Gamma^{1,2}_t))\notag\\  &+\sum_{\tilde{x}_t^{1,2}}P(Z_t=x_t^{1,2}|\pi_t^{1,2},\Gamma_t^{1,2})V_{t^+}(\tilde{x}^1_t,\tilde{x}^2_t)]
%\end{align}

\section{Generalized Communication Models}
\paragraph{Erasure Model}
In this model, agents lose packets or fail to communicate with probability $p_e$ even when one (or both) of the agents decides to communicate, i.e. when $M^{or}_t=1$. Based on the communication model described above we can define a variable $Z^{er}_t$ as:
\begin{equation} \label{erasure:model}
  Z^{er}_t=\begin{cases}
    X^{1,2}_t, ~ w.p. ~1-p_e & \text{if $M^{or}_t=1$}.\\
    \phi, ~~~~~ w.p. ~~~~p_e & \text{if $M^{or}_t=1$}.\\
    \phi, & \text{if $M^{or}_t=0$}.
  \end{cases}
\end{equation}
The information structure at each agent is given by 
\begin{equation}
    I^i_{t}=\{X^i_{1:t},U^i_{1:t-1},Z^{er}_{1:t-1},M^{1,2}_{1:t-1}\}
\end{equation}
Lemma \ref{LEM:INDEPEN} and Proposition \ref{PROP:ONE} hold true for the erasure model and are proved in Appendix \ref{proof:Erasure model}. The coordinator can update its beliefs on the agents’ states in a sequential manner as described in the following lemma.
\begin{lemma}\label{LEMM:UPDATE}
$\Pi^i_1$ is the prior belief ($P_{X^i_1}$) on the initial state  $X^i_1$  and for each $t \geq 1$, 
\begin{equation}
    \Pi^i_{t^+}=\eta^{er_i}_t( \Pi^i_{t},\Gamma^i_t,Z^{er}_t),
\end{equation}
\begin{equation}
    \Pi^i_{t+1}=\beta^{er_i}_t( \Pi^i_{t^+},\Lambda^i_{t}),
\end{equation}
where   $\eta^{er_i}_t, \beta^{er_i}_t$ are fixed functions derived from the system model. (We will use $\beta^{er}_t( \Pi^{1,2}_{t^+},\Lambda^{1,2}_{t})$ to denote the pair $\beta^{er_1}_t( \Pi^1_{t^+},\Lambda^1_{t}), \beta^{er_2}_t( \Pi^2_{t^+},\Lambda^2_{t})$. Similar notation will be used for the pair $\eta^{er_1}_t(\cdot), \eta^{er_2}_t(\cdot)$.)
%\red{Shouldn't these by $\beta^i_t$ and $\eta^i_t$? Is the $\beta_t$ notation in DP consistent with this lemma?}
\end{lemma}
\begin{proof}
See Appendix \ref{proof:Coordd}.
\end{proof}

Equation \eqref{eq:probz1} and \eqref{eq:probz2} for the erasure model can be written as:
\begin{align}
&P(Z^{er}_t=\phi|\Pi_t^{1,2},\Gamma_t^{1,2}) = \notag\\
   &~~ =\sum_{x^1, x^2}\mathds{1}_{\{\Gamma^1(x^{1})=0\}}\mathds{1}_{\{\Gamma^2(x^{2})=0\}}\Pi^1_t(x^1)\Pi^2_t(x^2)\notag\\
   &+\sum_{x^1, x^2}p_e\Big[\max(\Gamma^1(x^{1}),\Gamma^2(x^{2}))\Big]\Pi^1_t(x^1)\Pi^2_t(x^2). \label{eq:probz11}
\end{align}
Similarly, the probability that $Z^{er}_t = (x^1,x^2)$ is given as 
\begin{align}
    &P(Z^{er}_t=(x^1,x^2)|\Pi_t^{1,2},\Gamma_t^{1,2})\notag\\
    &~~=(1-p_e)\Big[\max(\Gamma^1(x^{1}),\Gamma^2(x^{2}))\Big]\Pi^1_t(x^1)\Pi^2_t(x^2). \label{eq:probz22}
 \end{align}
 
%With these modified probabilities and Lemma \ref{LEMM:UPDATE}, one can then use the dynamic program in Theorem \ref{thm:dp} to obtain the optimal value functions and strategies for the model with an erasure channel.
With these modified probabilities and Lemma \ref{LEMM:UPDATE}, we can write a dynamic program for the coordinator's POMDP problem. In the following theorem, $\pi^i$ denotes a general probability distribution on $\mathcal{X}^i$ and $\delta_{x^i}$ denotes a delta distribution centered at $x^i$.

\begin{theorem}\label{thmm:dp}
The value functions  for the coordinator's dynamic program as follows: $V_{T+1}(\pi^1,\pi^2):=0$ for all $\pi^1,\pi^2,$ and for  $t=T, \ldots, 2, 1,$ \\
 \begin{align}
  &V_{t^+}(\pi^1,\pi^2) := \notag \\
 & \min_{\lambda^1, \lambda^2} \Big[\Big(\sum_{x^{1,2}}c_t\big(x^{1,2},\lambda^1(x^{1}),\lambda^2(x^{2})\big)\pi^1(x^1)\pi^2(x^2)\Big) \notag\\
 &+V_{t+1}(\beta^{er}_t( \pi^{1,2},\lambda^{1,2}))\Big],
\end{align}
where $\beta^{er}_t$ is as described in Lemma \ref{LEMM:UPDATE}.
  \begin{align}
  &V_{t}(\pi^1,\pi^2) := \notag \\
 &\min_{\gamma^1, \gamma^2}\Big[\rho\sum_{x^{1,2}}\max(\gamma^1(x^{1}),\gamma^2(x^{2}))\pi^1(x^1)\pi^2(x^2) \notag \\ 
&+P(Z^{er}_t=\phi|\pi^{1,2},\gamma^{1,2}) V_{t^+}(\eta^{er}_t( \pi^{1,2},\gamma^{1,2}, \phi))\notag\\ &+\sum_{\tilde{x}^{1,2}}P(Z^{er}_t=\tilde{x}^{1,2}|\pi^{1,2},\gamma^{1,2})V_{t^+}(\delta_{\tilde{x}^1},\delta_{\tilde{x}^2})\Big], \label{dyn:update}
\end{align}
where $\eta^{er}_t$ is as described in Lemma \ref{LEMM:UPDATE} and $P(Z^{er}_t=\phi|\pi^{1,2},\gamma^{1,2})$, $P(Z^{er}_t=\tilde{x}^{1,2}|\pi^{1,2},\gamma^{1,2})$ are as described by \eqref{eq:probz11} - \eqref{eq:probz22}. The coordinator's optimal strategy is to pick the minimizing prescription pairs for each time and each $(\pi^1,\pi^2)$.
\end{theorem}
\begin{proof}
Since the coordinator's problem is a POMDP, it has a corresponding dynamic program and the value functions in the theorem can be obtained by simple manipulations of the POMDP dynamic program for the coordinator. 
\end{proof}
\paragraph{State-dependent communication cost}
Whenever agents decide to share their states with each other, they incur a state-dependent  cost $\rho(X_t)$ instead of a fixed cost $\rho$. 
%The cost objective changes to $\rho(X_t)\times$
The cost objective in this case is given by:
\begin{equation}\label{eq:cost2}
   \ee\Big[\sum_{t=1}^{T}c_t({X}_t,{U}_t)+\rho(X_t)\mathds{1}_{\{M^{or}_t=1\}}\Big].
\end{equation}

In order to find the optimal value function and strategies, one can use the dynamic program in Theorem \ref{thm:dp} with the following modification to \eqref{dyn:update}:
\begin{align}
  &V_{t}(\pi^1,\pi^2) := \notag \\
 &\min_{\gamma^1, \gamma^2}\Big[\sum_{x^{1,2}}\rho(x^{1,2})\max(\gamma^1(x^{1}),\gamma^2(x^{2}))\pi^1(x^1)\pi^2(x^2) \notag \\ 
&+P(Z_t=\phi|\pi^{1,2},\gamma^{1,2}) V_{t^+}(\eta_t( \pi^{1,2},\gamma^{1,2}, \phi))\notag\\ &+\sum_{\tilde{x}^{1,2}}P(Z_t=\tilde{x}^{1,2}|\pi^{1,2},\gamma^{1,2})V_{t^+}(\delta_{\tilde{x}^1},\delta_{\tilde{x}^2})\Big].
\end{align}

\section{ Agents with communication constraints}\label{sec:constraints}
%Include constraints on 
%\bdashlist
%\item min and max time between successive comm
%\item total number of comms
%\end{list}
%in the problem formulation. 

In this section, we consider an extension of the problem formulated in Section \ref{sec:problem_formulation} where we incorporate some constraints on the communication between agents. The underlying system model,  information structure and the total expected cost  are the same as in Section \ref{sec:problem_formulation}. But now the agents have constraints on when and how frequently they can communicate. Specifically, we consider  the following three constraints:
\begin{enumerate}
\item Minimum time between successive communications must be at least $s_{min}$ (where $s_{min} \geq 0$).
\item Maximum time between successive communications cannot exceed $s_{max}$ (where $s_{max} \geq s_{min}$).
\item The total number of communications over the time horizon $T$ cannot exceed $N$.
\end{enumerate}
The strategy optimization problem is to find communication and control strategies for the agents that minimize the expected cost in \eqref{eq:cost} while ensuring that the above three constraints are satisfied.  We assume that there is at least one choice of strategies for the agents for which the constraints are satisfied (i.e. the constrained problem is feasible). Note that our framework allows for some of the above three constraints to be absent (e.g. setting $s_{min}=0$ effectively removes the first constraint; setting $N=T$ effectively removes the third constraint). 

We can follow the methodology of Section \ref{Coordinator Results} for the constrained problem as well. The key difference is that in addition to the coordinator's beliefs on the agents' states, we will also need to keep track of (i) the time since the most recent communication (denoted by $S^a_t$), and (ii) the total number of  communications so far (denoted by $S^b_t$). The variables $S^a_t, S^b_t$ are used by the coordinator to make sure that the prescriptions it selects will not violate the constraints. For example, if $S^a_t < s_{min}$, the coordinator can only select the communication prescriptions  that map $\mathcal{X}^i$ to $0$ for each $i$ since this ensures that the first constraint will be satisfied. Similarly, if $S^a_t = s_{max}$, then the coordinator must select a pair of communication prescriptions that ensure that a communication happens at the current time. The following theorem describes the modified dynamic program for the coordinator in the constrained formulation.
\begin{theorem}
The value functions  for the coordinator's dynamic program are as follows: $V_{T+1}(\pi^1,\pi^2, s^a, s^b):=0$ for all $\pi^1,\pi^2, s^a,s^b,$ and for  $t=T, \ldots, 2, 1,$ \\
 \begin{align}
  &V_{t^+}(\pi^1,\pi^2, s^a, s^b) := \notag \\
 & \min_{\lambda^1, \lambda^2} \Big[\sum_{x^{1,2}}c_t\big(x^{1,2},\lambda^1(x^{1}),\lambda^2(x^{2})\big)\pi^1(x^1)\pi^2(x^2) \notag\\
 &+V_{t+1}(\beta_t( \pi^{1,2},\lambda^{1,2}),s^a,s^b)\Big],
\end{align}
where $\beta_t$ is as described in Lemma \ref{LEM:UPDATE}; and if $s_{min} \leq s^a < s_{max}$ and $s^b < N$
  \begin{align}
  &V_{t}(\pi^1,\pi^2,s^a,s^b) := \notag \\
 &\min_{\gamma^1, \gamma^2}\Big[\rho\sum_{x^{1,2}}\max(\gamma^1(x^{1}),\gamma^2(x^{2}))\pi^1(x^1)\pi^2(x^2) \notag \\ 
&+P(Z_t=\phi|\pi^{1,2},\gamma^{1,2}) V_{t^+}(\eta_t( \pi^{1,2},\gamma^{1,2}, \phi), s^a+1,s^b)\notag\\ &+\sum_{\tilde{x}^{1,2}}P(Z_t=\tilde{x}^{1,2}|\pi^{1,2},\gamma^{1,2})V_{t^+}(\delta_{\tilde{x}^1},\delta_{\tilde{x}^2},0,s^b+1)\Big], \label{eq:dp_gamma}
\end{align}
where $\eta_t$ is as described in Lemma \ref{LEM:UPDATE} and $P(Z_t=\phi|\pi^{1,2},\gamma^{1,2})$, $P(Z_t=\phi|\pi^{1,2},\gamma^{1,2})$ are as described by \eqref{eq:probz1} - \eqref{eq:probz2}. The coordinator's optimal strategy is to pick the minimizing prescription pairs for each time and each $(\pi^1,\pi^2)$.  If $s^b=N$ or if $s^a < s_{min}$, then the minimization over $\gamma^1,\gamma^2$ in \eqref{eq:dp_gamma} is replaced by simply setting $\gamma^1,\gamma^2$ to be the prescription that map all states to $0$. If $s^b < N$ and if $s^a = s_{max}$, then the minimization over $\gamma^1,\gamma^2$ in \eqref{eq:dp_gamma} is replaced by simply setting $\gamma^1,\gamma^2$ to be the prescription that map all states to $1$.   
\end{theorem}

%Write the corresponding pomdp dp.
%
%Can include the MDP version of the DP here.

\section{Experiments}
\textbf{Problem setup} Consider a system where there are two entities that are susceptible to attacks. Each entity has an associated defender that can either choose to do nothing (denoted by $\aleph$) or defend the system (denoted by $d$). Thus, the defenders are the decision-making agents in this model. The state $X_t^i \in \{0,1\}$ of agent $i$ represents whether or not entity $i$ is under attack at time $t$. We use 1 to denote the \emph{attack} state and 0 to denote the \emph{safe} (non-attack) state.  If entity $i$ is currently in the safe state, i.e. $X^i_t = 0$, then with probability $p^i_a$, the entity transitions to the attack state 1 (irrespective of the defender's action). When entity $i$ is under attack, i.e. $X_t^i = 1$, if the defender chooses to defend, then the entity transitions to the safe state 0 with probability $p^i_v$. If the defender chooses to do nothing, then the state does not change with probability 1, i.e. $X_{t+1}^i = X_t^{i}$. If both entities are in the safe state, then the cost incurred by the system is 0. If at least one entity is under attack, then the cost incurred is 20. Further, an additional cost of 150 is incurred if both defenders choose to defend at the same time (in any state). We consider a range of communication costs $\rho$. The transition and cost structures can be found in a tabular form in Tables \ref{trans} and \ref{costable}.

When an entity is under attack, the associated defender needs to defend the system at some point of time (if not immediately). Otherwise, the system will remain in the attack state perpetually. However, a heavy cost is incurred if both agents defend at the same time. Therefore, the agents must defend their respective entities in a coordinated manner. Communicating with each other can help in coordinating effectively. On the other hand, communicating all the time can also lead to a high communication cost. This tradeoff between communication and coordination can be balanced optimally using our approach discussed in Section \ref{Coordinator Results}.

Our model shares some similarities with the multi-access broadcast problem with limited buffers. Note that unlike in \citep{hansen2004dynamic}, the transmitters in our model can choose to communicate with each other before using the broadcast channel.

\begin{table}[t]
\caption{The transition probabilities $\prob[X^i_{t+1} \mid X^i_t, U^i_t]$. The rows correspond to the current state $X_t^i$ and the columns correspond to the action $U_t^i$. Each entry in the table is a distribution over the set \{0,1\}.}
\label{trans}
%\vskip 0.15in
\begin{center}
\begin{small}
\begin{sc}
\begin{tabular}{lccccr}
\toprule
 & $\aleph$ & $d$  \\
\midrule
$0$   & (1-$p^i_a$,$p_a^i$)& (1-$p^i_a$,$p_a^i$)\\
$1$   & (0,1)& ($p^i_v$,1-$p_v^i$)\\
\bottomrule
\end{tabular}
\end{sc}
\end{small}
\end{center}
\vskip -0.1in
\end{table}

\begin{table}[t]
\caption{The cost $c_t(X_t,U_t) = \vartheta^{t-1}c(x,u)$ where $\vartheta$ is the discount factor and the function $c(x,u)$ is shown below. Each row corresponds to a pair of actions $u := (u^{1},u^{2})$ and each column corresponds to a state $x:=(x^1,x^2)$.}
\label{costable}
%\vskip 0.15in
\begin{center}
\begin{small}
\begin{sc}
\begin{tabular}{lccccr}
\toprule
 & $(0,0)$ & $(0,1)$ & $(1,0)$ & $(1,1)$ \\
\midrule
$(\aleph, \aleph)$   & 0 & 20& 20& 20\\
$(\aleph,d)$   & 0& 20& 20 & 20\\
$(d,\aleph)$   & 0& 20& 20 & 20\\
$(d,d)$   & 150 & 170& 170& 170\\
\bottomrule
\end{tabular}
\end{sc}
\end{small}
\end{center}
\vskip -0.1in
\end{table}

\textbf{Implementation} In our experiments, the transition probabilities are time invariant and the cost is discounted as shown in Table \ref{costable}. The time horizon $T$ is very large. When the horizon is sufficiently large, the optimal discounted costs in finite-horizon and infinite-horizon cases are approximately the same. We therefore use an infinite-horizon POMDP solver to find approximately optimal finite horizon cost for the coordinator's POMDP. Since the agents alternate between communication and control (see equations \eqref{commact1} and \eqref{controlact1}), the coordinator's POMDP as described in  Section \ref{Coordinator Results} is not time-invariant. To convert it into a time-invariant POMDP, we introduce an additional binary state variable $X_t^{c}$. This variable represents whether the agents currently are in the communication phase or the control phase. The variable $X_t^c$ alternates between 0 and 1 in a deterministic manner. For agent $i$ in the communication phase, action $\aleph$ is interpreted as the no communication decision ($M_t^i = 0$) and all other actions (in this case only $d$) are interpreted as the communication decision ($M_t^i = 1$).

With this transformation, we can use any infinite-horizon POMDP solver to obtain approximately optimal solutions for our problem. In our experiments, we use the SARSOP solver \citep{kurniawati2008sarsop} that is available in the Julia POMDPs framework \citep{egorov2017pomdps}.
\begin{remark}
As discussed in Remark \ref{infiniteremark}, the procedure described above can be used to find optimal values and strategies even for the case when the time-horizon is infinite.
\end{remark}
\begin{remark}A major challenge in solving the coordinator's POMDP is that the prescription space is exponential in state size $\mathcal{X}^i$. When state spaces are large, even single stage problems are very difficult to solve. In such situations, one can use alternative approaches as in \citep{foerster2019bayesian} wherein the prescriptions are parameterized in a concise manner. However, these methods are not guaranteed to result in optimal solutions.
\end{remark}

\textbf{Results} We consider three strategies in our experiments: (i) the jointly optimal communication and control strategy computed using the coordinator's POMDP, (ii) agents never communicate with each other and the control strategy is optimized subject to this constraint, and (iii) agents always communicate with each other and the control strategy is optimized subject to this constraint. The expected cost associated with these three strategies is shown in Tables \ref{res1} and \ref{res2} under two different choices of model parameters. The approximation error achieved using the SARSOP solver is at most 0.001.

\begin{table}[t]
\caption{Optimal costs with initial state $X_1 = (0,0)$ under different communication strategies. The parameters $p^i_a = 0.3$ and $p^i_v = 0.6$, $i = 1,2$. The discount factor $\vartheta$ is $0.95$.}
\label{res1}
%\vskip 0.15in
\begin{center}
\begin{small}
\begin{sc}
\begin{tabular}{lccccr}
\toprule
 & Optimal & Never comm. & Always comm. \\
\midrule
$\rho = 0$   & 105.86 & 120.42& 105.86\\
$\rho = 1$   & 108.45 & 120.42& 116.12\\
$\rho = 2$   & 111.05 & 120.42& 126.37\\
$\rho = 4$   & 116.24 & 120.42& 146.89\\
$\rho = 8$   & 120.42 & 120.42& 187.96\\
\bottomrule
\end{tabular}
\end{sc}
\end{small}
\end{center}
\vskip -0.1in
\end{table}

\begin{table}[t]
\caption{Optimal costs with initial state $X_1 = (0,0)$ under different communication strategies. The parameters $p^1_a = 0.5$, $p^1_v = 0.95$, $p_a^2 = 0.1$ and $p_v^2 = 0.6$. The discount factor $\vartheta$ in this case is $0.99$.}
\label{res2}
%\vskip 0.15in
\begin{center}
\begin{small}
\begin{sc}
\begin{tabular}{lccccr}
\toprule
 & Optimal & Never comm. & Always comm. \\
\midrule
$\rho = 0$   & 452.18 & 538.13& 452.87\\
$\rho = 1$   & 459.05 & 538.13& 502.43\\
$\rho = 2$   & 465.93 & 538.13& 552.68\\
$\rho = 4$   & 479.64 & 538.13& 653.19\\
$\rho = 8$   & 504.26 & 538.13& 854.19\\
$\rho = 16$   & 537.86 & 538.13& 1256.21\\
$\rho = 32$   & 538.13 & 538.13& 2060.23\\
\bottomrule
\end{tabular}
\end{sc}
\end{small}
\end{center}
\vskip -0.1in
\end{table}

\section{Conclusion}
We considered  a multi-agent problem  where  agents can dynamically decide at each time step whether to share information with each other and incur the resulting communication cost. Our goal was to jointly design  agents' communication and control strategies in order to optimize the trade-off between communication costs and control objective. We  showed that agents can ignore a big part of their private information without compromising the system performance. We then  provided a common information approach based solution for the strategy optimization problem. Our approach relies on constructing a fictitious POMDP whose solution (obtained via a dynamic program) characterizes the optimal strategies for the agents. We also extended our solution to incorporate constraints on when and how frequently agents can communicate.

%\clearpage
%\bibliographystyle{icml2021}
\bibliography{refs}

\clearpage

\onecolumn
\appendix
%\input{AppendixA}

%%%%%PUT THIS BACK IN
\section{Proof of Lemma \ref{LEM:INDEPEN}}
\label{proof:CI}
We prove the lemma by induction. At $t=1$, before communication decisions are made, \eqref{eq:indepen} is trivially true since there is no common information at this point and the agents' initial states are independent. For \eqref{eq:indepen2} at $t=1$, recall that $c_{1^+} = (z_1, m_1)$. The left hand side of \eqref{eq:indepen2} for $t=1$ can be written as 

%\red{Below, $f$ should be $f^i_t$ for appropriate $i$ and $t$. Please fix.  Your proof is not very formal   and not always clear. I am fixing the base case for this lemma. Read it very carefully and write the induction step along similar lines.}

\begin{align}
& \prob({x}_1,{u}_{1}|z_{1}, {m}_{1})=\frac{\prob({x}_1,{u}_{1},z_{1},{m}_{1})}{\prob(z_{1},{m}_{1})}  \notag\\
&=\frac{\prob({u}_{1}|{x}_1,z_{1},{m}_{1})\prob(z_{1}|{x}_1,{m}_{1})\prob(m_{1}|x_1)\prob(x_1)}{\prob(z_{1},m_{1})} \notag \\
 &\frac{\mathds{1}_{(u^1_1=g^1_1(x^1_1,z_1,m_1))}\mathds{1}_{(u^2_1=g^2_1(x^2_1,z_1,m_1))}\prob(z_{1}|x_1,m_{1})\mathds{1}_{(m^2_1=f^2_1(x^2_1))}\mathds{1}_{(m^1_1=f^1_1(x^1_1))}\prob(x^1_1)\prob(x^2_1)}{\prob(z_{1},m_{1})}. \label{rn1:base}
\end{align}

The first term (corresponding to $i=1$) on the right hand side of \eqref{eq:indepen2} for $t=1$ can be written as 
\begin{align}
     &\prob(x^1_1,u^1_{1}|z_{1},m_{1})=\frac{\prob(x^1_1,u^1_{1},z_{1},m_{1})}{\prob(z_{1},m_{1})}\notag\\
     &=\frac{\prob(u^1_{1}|x^1_1,z_{1},m_{1})\prob(z_{1}|x^1_1,m_{1})\prob(m_{1}|x^1_1)\prob(x^1_1)}{\prob(z_{1},m_{1})}\notag\\
    &=\frac{\mathds{1}_{(u^1_1=g^1_1(x^1_1,z_1,m_1))}\prob(z_{1}|x^1_1,m_{1})\mathds{1}_{(m^1_1=f^1_1(x^1_1))}\prob(m^2_{1}|x^1_1)\prob(x^1_1)}{\prob(z_{1},m_{1})}. \label{after1:comm}
\end{align}
Similarly, the second term (corresponding to $i=2$) on the right hand side of \eqref{eq:indepen2}  for $t=1$ can be written as
\begin{equation}\label{after2:comm}
    \frac{\mathds{1}_{(u^2_1=g^2_1(x^2_1,z_1,m_1))}\prob(z_{1}|x^2_1,m_{1})\mathds{1}_{(m^2_1=f^2_1(x^2_1))}\prob(m^1_{1}|x^2_1)\prob(x^2_1)}{\prob(z_{1},m_{1})}.
\end{equation}
Comparing \eqref{rn1:base}, \eqref{after1:comm} and \eqref{after2:comm}, it is clear that  we just need to prove that
\begin{align}
\frac{\prob(z_{1}|x_1,m_{1})}{\prob(z_{1},m_{1})}  = \frac{\prob(z_{1}|x^1_1,m_{1})\prob(m^2_{1}|x^1_1)}{\prob(z_{1},m_{1})} \times \frac{\prob(z_{1}|x^2_1,m_{1})\prob(m^1_{1}|x^2_1)}{\prob(z_{1},m_{1})} \label{eq:new_1}
\end{align}
in order to establish \eqref{eq:indepen2} for $t=1$. We consider two cases:

\textbf{Case I: } $Z_1=\phi$ and $M_1 = (0,0)$. In this case, the left hand side of \eqref{eq:new_1} can be written as 
\begin{align}
\frac{\prob(Z_{1}= \phi|X_1=x_1,M_{1}=(0,0))}{\prob(Z_{1}=\phi,M_{1} = (0,0))} = \frac{1}{\prob(M^1_{1} = 0)\prob(M^2_1=0)} \label{eq:new6}
\end{align}
Similarly, the right hand side of  \eqref{eq:new_1} can be written as 
\begin{align}
 &\frac{P(M^2_1=0|X^1_1=x^1_1)}{\prob(M^1_{1} = 0)\prob(M^2_1=0)} \times \frac{P(M^1_1=0|X^2_1=x^2_1)}{\prob(M^1_{1} = 0)\prob(M^2_1=0)} \notag \\
 &= \frac{P(M^2_1=0)}{\prob(M^1_{1} = 0)\prob(M^2_1=0)} \times \frac{P(M^1_1=0)}{\prob(M^1_{1} = 0)\prob(M^2_1=0)} \notag \\
 &= \frac{1}{\prob(M^1_{1} = 0)\prob(M^2_1=0)} \label{eq:new5}
\end{align}
where we used the fact that $M^i_1$ and $X^j_1$ are independent for $i \neq j$. Comparing \eqref{eq:new6} and \eqref{eq:new5} establishes \eqref{eq:new_1} for this case.

\textbf{Case II: } $Z_1=(\tilde{x}^1_1, \tilde{x}^2_1) $ and $M_1 = (m^1_1,m^2_1) \neq (0,0)$. In this case, the left hand side of \eqref{eq:new_1} can be written as 
\begin{align}
&\frac{\prob(Z_{1}= (\tilde{x}^1_1, \tilde{x}^2_1)|X_1=x_1,M_{1}=(m^1_1,m^2_1))}{\prob(Z_{1}= (\tilde{x}^1_1, \tilde{x}^2_1),M_{1}=(m^1_1,m^2_1))} = \frac{\mathds{1}_{\{x_1=(\tilde{x}^1_1, \tilde{x}^2_1)\}}}{\prob(X_{1}= (\tilde{x}^1_1, \tilde{x}^2_1),M_{1}=(m^1_1,m^2_1))} \notag \\
&= \frac{\mathds{1}_{\{x_1=(\tilde{x}^1_1, \tilde{x}^2_1)\}}}{\prob(X_{1}= (\tilde{x}^1_1, \tilde{x}^2_1))\mathds{1}_{(m^1_1=f^1_1(\tilde{x}^1_1))}\mathds{1}_{(m^2_1=f^2_1(\tilde{x}^2_1))}} \label{eq:new2}
\end{align}
Consider the first term on the right hand side of \eqref{eq:new_1}. It can be written as
\begin{align}
&\frac{\prob(Z_{1} = (\tilde{x}^1_1, \tilde{x}^2_1) |X^1_1 = x^1_1, M_1 =m_{1})\prob(M^2_1 = m^2_{1}|X^1_1=x^1_1)}{\prob(Z_{1}= (\tilde{x}^1_1, \tilde{x}^2_1),M_{1}=(m^1_1,m^2_1))} \notag \\
& =  \frac{\mathds{1}_{\{x^1_1=\tilde{x}^1_1\}}\prob(X^2_1 = \tilde{x}^2_1|M^2_1 = m^2_1)\prob(M^2_1 = m^2_{1}) }{\prob(X_{1}= (\tilde{x}^1_1, \tilde{x}^2_1))\mathds{1}_{(m^1_1=f^1_1(\tilde{x}^1_1))}\mathds{1}_{(m^2_1=f^2_1(\tilde{x}^2_1))}} \notag \\
&= \frac{\mathds{1}_{\{x^1_1=\tilde{x}^1_1\}}\prob(X^2_1 = \tilde{x}^2_1, M^2_1 = m^2_1) }{\prob(X_{1}= (\tilde{x}^1_1, \tilde{x}^2_1))\mathds{1}_{(m^1_1=f^1_1(\tilde{x}^1_1))}\mathds{1}_{(m^2_1=f^2_1(\tilde{x}^2_1))}} \notag \\
&= \frac{\mathds{1}_{\{x^1_1=\tilde{x}^1_1\}}\prob(X^2_1 = \tilde{x}^2_1)\mathds{1}_{(m^2_1=f^2_1(\tilde{x}^2_1))} }{\prob(X_{1}= (\tilde{x}^1_1, \tilde{x}^2_1))\mathds{1}_{(m^1_1=f^1_1(\tilde{x}^1_1))}\mathds{1}_{(m^2_1=f^2_1(\tilde{x}^2_1))}} \notag \\
&= \frac{\mathds{1}_{\{x^1_1=\tilde{x}^1_1\}} }{\prob(X^1_{1}= \tilde{x}^1_1)\mathds{1}_{(m^1_1=f^1_1(\tilde{x}^1_1))}} \label{eq:new3}
\end{align} 
where we used the fact that $(X^i_1,M^i_1)$ are independent of  $X^j_1$  for $i \neq j$. Similarly, the second term on the right hand side of  \eqref{eq:new_1} can be written as 
\begin{align}
&\frac{\prob(Z_{1} = (\tilde{x}^1_1, \tilde{x}^2_1) |X^2_1 = x^1_1, M_1 =m_{1})\prob(M^1_1 = m^1_{1}|X^2_1=x^2_1)}{\prob(Z_{1}= (\tilde{x}^1_1, \tilde{x}^2_1),M_{1}=(m^1_1,m^1_2))} \notag \\
&= \frac{\mathds{1}_{\{x^2_1=\tilde{x}^2_1\}} }{\prob(X^2_{1}= \tilde{x}^2_1)\mathds{1}_{(m^2_1=f^2_1(\tilde{x}^2_1))}} \label{eq:new4}
\end{align} 
Comparing \eqref{eq:new2}, \eqref{eq:new3} and \eqref{eq:new4} establishes \eqref{eq:new_1} for this case.

\textbf{Induction step:} Assuming that Lemma \ref{LEM:INDEPEN} holds for  $t$, we need to show that it holds for $t+1$. At time $t+1$, before 
communication decisions are made, left hand side of equation \eqref{eq:indepen} can be written as
\begin{align}
& \prob(x_{1:t+1},u_{1:t}|z_{1:t},m_{1:t})=\prob(x^{1,2}_{t+1}|x_{1:t},u_{1:t},z_{1:t},m_{1:t})\prob(x_{1:t},u_{1:t}|z_{1:t},m_{1:t})  \notag\\
&=\prob(x^{2}_{t+1}|x^{1}_{t+1},x_{1:t},u_{1:t},z_{1:t},m_{1:t})\prob(x^1_{t+1}|x_{1:t},u_{1:t},z_{1:t},m_{1:t})P(x_{1:t},u_{1:t}|z_{1:t},m_{1:t}) \notag \\
 &=\prob(x^{2}_{t+1}|x^2_{t},u^2_{t})\prob(x^{1}_{t+1}|x^1_{t},u^1_{t})\prob(x^1_{1:t},u^1_{1:t}|z_{1:t},m_{1:t})\prob(x^2_{1:t},u^2_{1:t}|z_{1:t},m_{1:t}). \label{ln1:inductive}
\end{align}

The first term (corresponding to $i=1$) on the right hand side of \eqref{eq:indepen} for $t+1$ can be written as 
\begin{align}
     &\prob(x^1_{1:t+1},u^1_{1:t}|z_{1:t},m_{1:t})=\prob(x^{1}_{t+1}|x^1_{1:t},u^1_{1:t},z_{1:t},m_{1:t})P(x^1_{1:t},u^1_{1:t}|z_{1:t},m_{1:t})\notag\\
     &=\prob(x^{1}_{t+1}|x^1_{t},u^1_{t})P(x^1_{1:t},u^1_{1:t}|z_{1:t},m_{1:t}). \label{rn1:inductive}
\end{align}
Similarly for $i=2$,
\begin{equation}\label{rn2:inductive}
    P(x^2_{1:t+1},u^2_{1:t}|z_{1:t},m_{1:t})=\prob(x^{2}_{t+1}|x^2_{t},u^2_{t})P(x^2_{1:t},u^2_{1:t}|z_{1:t},m_{1:t}).
\end{equation}
Comparing \eqref{ln1:inductive}, \eqref{rn1:inductive} and \eqref{rn2:inductive} establishes \eqref{eq:indepen}   for time $t+1$.

For \eqref{eq:indepen2} at $t+1$, recall that $c_{{t+1}^+} = (z_{1:t+1},m_{1:t+1})$. The left hand side of \eqref{eq:indepen2} for $t+1$ can be written as 
\begin{align}
&\prob(x_{1:t+1},u_{1:t+1}|z_{1:t+1},m_{1:t+1})=
\frac{\prob(x_{1:t+1},u_{1:t+1},z_{t+1},m_{t+1}| z_{1:t},m_{1:t})}{\prob(z_{t+1},m_{t+1} | z_{1:t},m_{1:t})} \notag\\
&=\prob(u_{t+1}|x_{1:t+1},u_{1:t},z_{1:t+1},m_{1:t+1})\times\notag\\
&\frac{\prob(z_{t+1}|x_{1:t+1},u_{1:t},z_{1:t},m_{1:t+1})\prob(m_{t+1}|x_{1:t+1},u_{1:t},z_{1:t},m_{1:t})\prob(x_{1:t+1},u_{1:t}|z_{1:t},m_{1:t})}{\prob(z_{t+1},m_{t+1} | z_{1:t},m_{1:t})} \notag \\
 &=\mathds{1}_{(u^2_{t+1}=g^2_{t+1}(x^2_{1:t+1},u^2_{1:t},z_{1:t+1},m_{1:t+1}))}\mathds{1}_{(u^1_{t+1}=g^1_{t+1}(x^1_{1:t+1},u^1_{1:t},z_{1:t+1},m_{1:t+1}))}\prob(z_{t+1}|x_{1:t+1},u_{1:t},z_{1:t},m_{1:t+1})\times \notag\\
&\frac{\mathds{1}_{(m^2_{t+1}=f^2_{t+1}(x^2_{1:t+1},u^2_{1:t},z_{1:t},m_{1:t}))}\mathds{1}_{(m^1_{t+1}=f^1_{t+1}(x^1_{1:t+1},u^1_{1:t},z_{1:t},m_{1:t}))}\prob(x^1_{1:t+1},u^1_{1:t}|z_{1:t},m_{1:t})\prob(x^2_{1:t+1},u^2_{1:t}|z_{1:t},m_{1:t})}{\prob(z_{t+1},m_{t+1} | z_{1:t},m_{1:t})}. \label{rhs2:inductive}
\end{align}

The first term (corresponding to $i=1$) on the right hand side of \eqref{eq:indepen2} for $t+1$ can be written as 
\begin{align}
     &\prob(x^1_{1:t+1},u^1_{1:t+1}|z_{1:t+1},m_{1:t+1})=\frac{\prob(x^1_{1:t+1},u^1_{1:t+1},z_{t+1},m_{t+1}| z_{1:t},m_{1:t})}{\prob(z_{t+1},m_{t+1} | z_{1:t},m_{1:t})}\notag\\
     &=\prob(u^1_{t+1}|x^1_{1:t+1},u^1_{1:t},z_{1:t+1},m_{1:t+1}) \times \notag\\
     &\frac{\prob(z_{t+1}|x^1_{1:t+1},u^1_{1:t},z_{1:t},m_{1:t+1})\prob(m_{t+1}|x^1_{1:t+1},u^1_{1:t},z_{1:t},m_{1:t})\prob(x^1_{1:t+1},u^1_{1:t}|z_{1:t},m_{1:t})}{\prob(z_{t+1},m_{t+1} | z_{1:t},m_{1:t})}\notag\\
    &=\mathds{1}_{(u^1_{t+1}=g^1_{t+1}(x^1_{1:t+1},u^1_{1:t},z_{1:t+1},m_{1:t+1}))}\prob(z_{t+1}|x^1_{1:t+1},u^1_{1:t},z_{1:t},m_{1:t+1})\times \notag\\
&\frac{\mathds{1}_{(m^1_{t+1}=f^1_{t+1}(x^1_{1:t+1},u^1_{1:t},z_{1:t},m_{1:t})}\prob(m^2_{t+1}|x^1_{1:t+1},u^1_{1:t},z_{1:t},m_{1:t})\prob(x^1_{1:t+1},u^1_{1:t}|z_{1:t},m_{1:t})}{\prob(z_{t+1},m_{t+1} | z_{1:t},m_{1:t})}. \label{rhs2:comm}
\end{align}
Similarly, the second  term (corresponding to $i=2$) on the right hand side of \eqref{eq:indepen2} for $t+1$ can be written as 
\begin{align}\label{rhs2:t+1}
&\mathds{1}_{(u^2_{t+1}=g^2_{t+1}(x^2_{1:t+1},u^2_{1:t},z{1:t+1},m_{1:t+1}))}\prob(z_{t+1}|x^2_{1:t+1},u^2_{1:t},z_{1:t},m_{1:t+1})\times\notag\\
&\frac{\prob(m^1_{t+1}|x^2_{1:t+1},u^2_{1:t},z_{1:t},m_{1:t})\mathds{1}_{(m^2_{t+1}=f^2_{t+1}(x^2_{1:t+1},u^2_{1:t},z_{1:t},,m_{1:t})}\prob(x^2_{1:t+1},u^2_{1:t}|z_{1:t},m_{1:t})}{\prob(z_{t+1},m_{t+1} | z_{1:t},m_{1:t})}.
\end{align}

Comparing \eqref{rhs2:inductive}, \eqref{rhs2:comm} and \eqref{rhs2:t+1}, it is clear that  we just need to prove that
\begin{align}
\frac{\prob(z_{t+1}|x_{1:t+1},u_{1:t},z_{1:t},m_{1:t+1})}{\prob(z_{t+1},m_{t+1} | z_{1:t},m_{1:t})} &= \frac{\prob(z_{t+1}|x^1_{1:t+1},u^1_{1:t},z_{1:t},m_{1:t+1})\prob(m^2_{t+1}|x^1_{1:t+1},u^1_{1:t},z_{1:t},m_{1:t})}{\prob(z_{t+1},m_{t+1} | z_{1:t},m_{1:t})} \times \notag\\ &\frac{\prob(z_{t+1}|x^2_{1:t+1},u^2_{1:t},z_{1:t},m_{1:t+1})\prob(m^1_{t+1}|x^2_{1:t+1},u^2_{1:t},z_{1:t},m_{1:t})}{\prob(z_{t+1},m_{t+1} | z_{1:t},m_{1:t})}. \label{eq:new_11}
\end{align}
in order to establish \eqref{eq:indepen2} for $t+1$. We consider two cases:

\textbf{Case I: } $Z_{t+1}=\phi$ and $M_{t+1} = (0,0)$. In this case, the left hand side of \eqref{eq:new_11} can be written as

\begin{align}
\frac{\prob(Z_{t+1}= \phi|x_{1:t+1},u_{1:t},z_{1:t},m_{1:t},M_{t+1}=(0,0))}{\prob(Z_{t+1}=\phi,M_{t+1} = (0,0)|z_{1:t},m_{1:t})} = \frac{1}{\prob(M^1_{t+1} = 0|z_{1:t},m_{1:t})\prob(M^2_{t+1}=0|z_{1:t},m_{1:t})} \label{eq:new66}
\end{align}
Similarly, the right hand side of  \eqref{eq:new_11} can be written as 
\begin{align}
 &\frac{\prob(M^2_{t+1}=0|x^1_{1:t+1},u^1_{1:t},z_{1:t},m_{1:t})}{\prob(M^1_{t+1} = 0|z_{1:t},m_{1:t})\prob(M^2_{t+1}=0|z_{1:t},m_{1:t})} \times \frac{\prob(M^1_{t+1}=0|x^2_{1:t+1},u^2_{1:t},z_{1:t},m_{1:t})}{\prob(M^1_{t+1} = 0|z_{1:t},m_{1:t})\prob(M^2_{t+1}=0|z_{1:t},m_{1:t})} \notag \\
 &= \frac{\prob(M^2_{t+1}=0|z_{1:t},m_{1:t})}{\prob(M^1_{t+1} = 0|z_{1:t},m_{1:t})\prob(M^2_{t+1}=0|z_{1:t},m_{1:t})} \times \frac{\prob(M^1_{t+1} = 0|z_{1:t},m_{1:t})}{\prob(M^1_{t+1} = 0|z_{1:t},m_{1:t})\prob(M^2_{t+1}=0|z_{1:t},m_{1:t})} \notag \\
 &=\frac{1}{\prob(M^1_{t+1} = 0|z_{1:t},m_{1:t})\prob(M^2_{t+1}=0|z_{1:t},m_{1:t})} \label{eq:new55}
\end{align}
where we used the fact that $M^i_{t+1}$ and $X^j_{t+1}$ are independent for $i \neq j$. Comparing \eqref{eq:new66} and \eqref{eq:new55} establishes \eqref{eq:new_11} for this case.

\textbf{Case II: } $Z_{t+1}=(\tilde{x}^1_{t+1}, \tilde{x}^2_{t+1}) $ and $M_{t+1} = (m^1_{t+1},m^2_{t+1}) \neq (0,0)$. In this case, the left hand side of \eqref{eq:new_11} can be written as

\begin{align}
&\frac{\prob(Z_{t+1}= (\tilde{x}^1_{t+1}, \tilde{x}^2_{t+1})|x_{1:t+1},u_{1:t},z_{1:t},m_{1:t},M_{t+1}=(m^1_{t+1},m^2_{t+1}))}{\prob(Z_{t+1}=(\tilde{x}^1_{t+1}, \tilde{x}^2_{t+1}),M_{t+1} = (m^1_{t+1},m^2_{t+1})|z_{1:t},m_{1:t})} \notag\\
&= \frac{\mathds{1}_{\{x_{t+1}=(\tilde{x}^1_{t+1}, \tilde{x}^2_{t+1})\}}}{\prob(X_{t+1}=(\tilde{x}^1_{t+1}, \tilde{x}^2_{t+1}),M_{t+1} = (m^1_{t+1},m^2_{t+1})|z_{1:t},m_{1:t})} \notag \\
&= \frac{\mathds{1}_{\{x_{t+1}=(\tilde{x}^1_{t+1}, \tilde{x}^2_{t+1})\}}}{\prob(X^1_{t+1}=\tilde{x}^1_{t+1}, M^1_{t+1}= m^1_{t+1}|z_{1:t},m_{1:t}) \prob(X^2_{t+1}=\tilde{x}^2_{t+1}, M^2_{t+1} =m^2_{t+1}|z_{1:t},m_{1:t})} \label{eq:new22}
\end{align}
where we used  \eqref{eq:indepen}   for time $t+1$.
%&= \frac{\mathds{1}_{\{x_{t+1}=(\tilde{x}^1_{t+1}, \tilde{x}^2_{t+1})\}}}{\prob(X_{t+1}=(\tilde{x}^1_{t+1}, \tilde{x}^2_{t+1}))P(m^1_{t+1}|{x}^1_{t+1},z_{1:t},m_{1:t})P(m^2_{t+1}|{x}^2_{t+1},z_{1:t},m_{1:t})}

Consider the first term on the right hand side of \eqref{eq:new_11}. It can be written as
\begin{align}
&\frac{\prob(Z_{t+1}=(\tilde{x}^1_{t+1}, \tilde{x}^2_{t+1})|x^1_{1:t+1},u^1_{1:t},z_{1:t},m_{1:t},M_{t+1}=(m^1_{t+1},m^2_{t+1}))\prob(M^2_{t+1}=m^2_{t+1}|x^1_{1:t+1},u^1_{1:t},z_{1:t},m_{1:t})}{\prob(Z_{t+1}=(\tilde{x}^1_{t+1}, \tilde{x}^2_{t+1}),M_{t+1} = (m^1_{t+1},m^2_{t+1})|z_{1:t},m_{1:t})} \notag\\
& =  \frac{\mathds{1}_{\{x^1_{t+1}=\tilde{x}^1_{t+1}\}}\prob({X}^2_{t+1}=\tilde{x}^2_{t+1}|z_{1:t},m_{1:t},m^2_{t+1})\prob(m^2_{t+1}|z_{1:t},m_{1:t}) }{\prob(X^1_{t+1}=\tilde{x}^1_{t+1}, M^1_{t+1}= m^1_{t+1}|z_{1:t},m_{1:t}) \prob(X^2_{t+1}=\tilde{x}^2_{t+1}, M^2_{t+1} =m^2_{t+1}|z_{1:t},m_{1:t})} \notag \\
& =  \frac{\mathds{1}_{\{x^1_{t+1}=\tilde{x}^1_{t+1}\}}\prob({X}^2_{t+1}=\tilde{x}^2_{t+1},M^2_{t+1}=m^2_{t+1}|z_{1:t},m_{1:t}) }{\prob(X^1_{t+1}=\tilde{x}^1_{t+1}, M^1_{t+1}= m^1_{t+1}|z_{1:t},m_{1:t}) \prob(X^2_{t+1}=\tilde{x}^2_{t+1}, M^2_{t+1} =m^2_{t+1}|z_{1:t},m_{1:t})} \notag\\
&=\frac{\mathds{1}_{\{x^1_{t+1}=\tilde{x}^1_{t+1}\}} }{\prob(X^1_{t+1}=\tilde{x}^1_{t+1}, M^1_{t+1}= m^1_{t+1}|z_{1:t},m_{1:t})}
\label{eq:new33}
\end{align} 
where we used the fact that $(X^i_{t+1},M^i_{t+1})$ are independent of  $X^j_{t+1}$  for $i \neq j$. Similarly, the second term on the right hand side of  \eqref{eq:new_11} can be written as 
\begin{align}
&\frac{\mathds{1}_{\{x^2_{t+1}=\tilde{x}^2_{t+1}\}} }{\prob(X^2_{t+1}=\tilde{x}^2_{t+1}, M^2_{t+1}= m^2_{t+1}|z_{1:t},m_{1:t})}\label{eq:new44}
\end{align} 
Comparing \eqref{eq:new22}, \eqref{eq:new33} and \eqref{eq:new44} establishes \eqref{eq:new_11} for this case.

\section{Proof of Proposition 1} \label{proof:prop1}
%\blue{Moved from main text:}
%\begin{lemma}
 We will prove the result for agent $i$. Throughout this proof, we fix agent $-i's$ communication and control strategies to be  $f^{-i}$,$g^{-i}$ (where f$^{-i}$,$g^{-i}$ are arbitrarily chosen). Define $R^i_t=(X^i_{t},Z_{1:t-1},M^{1,2}_{1:t-1})$ and $R^i_{t^+}=(X^i_{t},Z_{1:t},M^{1,2}_{1:t})$. Our proof will rely on  the following two facts:
 
  \textbf{Fact 1}: $\{R^i_1, R^i_{1^+}, R^i_2, R^i_{^2+}, ....R^i_T, R^i_{T^+}\}$ is a controlled Markov process for agent $i$. 
  More precisely, for any strategy choice $f^i, g^i$ of agent $i$, 
%Write (A) and (B),
\begin{align}
&\prob(R^i_{t^+}=\tilde{r}^i_{t^+}|R^i_{1:t}=r^i_{1:t},M^i_{1:t}=m^i_{1:t})=\prob(R^i_{t^+}=\tilde{r}^i_{t^+}|R^i_{t}=r^i_{t},M^i_{t}=m^i_{t}) \label{eq:mdpA} \\
& \prob(R^{i}_{t+1}=\tilde{r}^{i}_{t+1}|R^i_{1:{t^+}}=r^i_{1:{t^+}},U^i_{1:t}=u^i_{1:t})=\prob(R^{i}_{t+1}=\tilde{r}^{i}_{t+1}|R^i_{{t^+}}=r^i_{{t^+}},U^i_{t}=u^i_{t}) \label{eq:mdpB}
\end{align}
where the probabilities on the right hand side of  \eqref{eq:mdpA} and \eqref{eq:mdpB} do not depend on $f^i,g^i$.

  \textbf{Fact 2}: The costs at time $t$ satisfy
%%Write LHS of C = k^i_t(R^i_t, M^i_t)
%%Write LHS of D = k^i_{t^+}(R^i_t^+, U^i_t),
\begin{align}
& \ee[\rho\mathds{1}_{(M^{or}_t=1)}|R^i_{1:t}=r^i_{1:t},M^i_{1:t}=m^i_{1:t}]= \kappa^i_t(r^i_t, m^i_t) \label{eq:mdpC}\\
& \ee[c_t({X}_t,{U}_t)|R^i_{1:{t^+}}=r^i_{1:{t^+}},U^i_{1:t}=u^i_{1:t}] =\kappa^i_{t^+}(r^i_{t^+}, u^i_t) \label{eq:mdpD}
 \end{align}
 %(R^i_{t^+}, U^i_t)
 %=\k^i_{t^+}(R^i_t^+, U^i_t) 
  where the functions $\kappa^i_t, \kappa^i_{t^+}$ in \eqref{eq:mdpC} and \eqref{eq:mdpD}  do not depend on $f^i, g^i$.

Suppose that Facts 1 and 2 are true. Then, the strategy optimization problem for agent $i$ can be viewed as a MDP over $2T$ time steps (i.e. time steps $1, 1^+, 2, 2^+, \ldots, T, T^+$) with $R^i_t$ and $M^i_t$ as the state and action at time $t$; and $R^i_{t^+}$ and $U^i_t$ as the state and action for time $t^+$. Note that at time $t$, agent $i$ observes $R^i_t$, selects $M^i_t$ and the ``state'' transitions to $R^i_{t^+}$ according to Markovian dynamics (see \eqref{eq:mdpA}). Similarly, at time $t^+$, agent $i$ observes $R^i_{t^+}$, selects $U^i_t$ and the ``state'' transitions to $R^i_{t+1}$ according to Markovian dynamics (see \eqref{eq:mdpB}). Further, from agent $i$'s perspective, the cost at time $t$ depends on the current state and action (i.e. $R^i_t$ and $M^i_t$, see \eqref{eq:mdpC}) and the cost at time $t^+$ depends on the state and action at $t^+$ (i.e. $R^i_{t^+}$ and $U^i_t$, see \eqref{eq:mdpD}).  It then follows from standard MDP results that agent $i$'s strategy should be of the form: 
\begin{equation} 
   M^i_t=  \bar{f}^i_t(R^i_{t}) =  \bar{f}^i_t(X^i_{t},Z_{1:t-1},M^{1,2}_{1:t-1}),
\end{equation}
%M^i_t= \bar{f}^i_t(X^i_{t},Z_{1:t-1},M^{1,2}_{1:t-1})
% U^i_t= \bar{g}^i_t(X^i_{t},Z_{1:t},M^{1,2}_{1:t})
\begin{equation}  
   U^i_t=  \bar{g}^i_t(R_{t^+}) = \bar{g}^i_t(X^i_{t}, Z_{1:t},M^{1,2}_{1:t}),
\end{equation}
which establishes the result of the proposition (recall that $C_t = (Z_{1:t-1},M^{1,2}_{1:t-1})$ and $C_{t^+}=(Z_{1:t},M^{1,2}_{1:t}$).

We now prove Facts 1 and 2. 

(i) Let $\tilde{r}^i_{t^+} = ({x}^i_t,{z}_{1:t},{m}_{1:t})$ and $r^i_{1:t} = (x^i_{1:t},z_{1:t-1},m_{1:t-1})$. Then, the left hand side of \eqref{eq:mdpA} can be written as 
\begin{align}
 &\prob(R^i_{t^+}=({x}^i_t,{z}_{1:t},{m}_{1:t})|R^i_{1:t}=(x^i_{1:t},z_{1:t-1},m_{1:t-1}),M^i_{1:t}=m^i_{1:t})\notag\\
    &=\prob({Z}_t={z}_t|x^i_{1:t},z_{1:t-1},m_{1:t})\prob({M}^{-i}_t={m}^{-i}_t|x^i_{1:t},z_{1:t-1},m_{1:t-1},m^i_{t}) \notag\\
    &= \prob({Z}_t={z}_t|x^i_{1:t},z_{1:t-1},m_{1:t})\prob({M}^{-i}_t={m}^{-i}_t|z_{1:t-1},m_{1:t-1},m^i_{t})  \label{eq:mark4}
\end{align}
where \eqref{eq:mark4} follows from the conditional independence property of Lemma \ref{LEM:INDEPEN}.
 We can further simplify the first term in  \eqref{eq:mark4} for different cases as follows: 
 
\textbf{Case I:} when $Z_t=\tilde{x}^{1,2}_t$ and ${M}_t=(m^1_t,m^2_t) \neq (0,0)$
\begin{align}
    \prob(Z_t=\tilde{x}^{1,2}_t|x^1_{1:t},z_{1:t-1},m_{1:t})=\mathds{1}_{(\tilde{x}^i_{t}={x}^i_{t})}\prob(\tilde{x}^{-i}_t|z_{1:t-1},m_{1:t}) \label{zt:A1}
\end{align}
\textbf{Case II:} when $Z_t=\phi$ and ${M}_t=(0,0)$
\begin{align}
    \prob(Z_t=\phi|x^1_{1:t},z_{1:t-1},m_{1:t})=1 \label{zt:A2}
\end{align}
We note that in both  cases above $x^i_{1:t-1}$ does not affect the probability.  Further, the probabilities in the two cases do not depend on agent $i$'s strategy.

Repeating the above steps for the right hand side of \eqref{eq:mdpA} establishes that the two sides of  \eqref{eq:mdpA} are equal.

(ii) \eqref{eq:mdpB} is a direct consequence of the Markovian state dynamics of agent $i$.

(iii) In \eqref{eq:mdpC}, it is straightforward to see that if $m^i_t =1$, then the left hand side is simply $\rho$. If, on the other hand, $m^i_t =0$, then the left hand side of \eqref{eq:mdpC} can be written as 
\begin{align}
&\rho\prob(M^{-i}_t =1|r^i_{1:t},m^i_{1:t}) = \rho\prob({M}^{-i}_t=1|x^i_{1:t},z_{1:t-1},m_{1:t-1},m^i_{t}) \notag \\
&=\rho\prob({M}^{-i}_t=1|z_{1:t-1},m_{1:t-1},m^i_{t}) \label{eq:markA1}
\end{align}
where \eqref{eq:markA1} follows from the conditional independence property of Lemma \ref{LEM:INDEPEN}. The right hand side of \eqref{eq:markA1} is a function only of $r^i_t$ and $m^i_t$ and does not depend on agent $i$'s strategy. This completes the proof of \eqref{eq:mdpC}. 

(iv) To prove \eqref{eq:mdpD}, it suffices to show that 
\begin{align}
 \prob(x^{-i}_t,u^{-i}_t|(x^i_{1:t},z_{1:t},m_{1:t}),u^i_{1:t})=  \prob(x^{-i}_t,u^{-i}_t|(x^i_{t},z_{1:t},m_{1:t}),u^i_{t}) \label{eq:markA2}
\end{align}
\eqref{eq:markA2} follows from the conditional independence property of Lemma \ref{LEM:INDEPEN}.

%\input{AN_AppC.tex}
%%%Dhruva: Use this file for editing  Appendix C%%%

\section{Proof of Lemma \ref{LEM:UPDATE}}\label{proof:Coord}
Recall that at the beginning of time $t$, the common information is given by $C_{t}:=(Z_{1:t-1},M^{1,2}_{1:t-1})$ (see \eqref{commoninfo}).
At the end of time $t$, i.e. after the communication decisions are made at time $t$, the common information is given by $C_{t^+}:=(Z_{1:t},M^{1,2}_{1:t})$ (see \eqref{commoninfoplus}). Let $c_{t} := (z_{1:t-1},m_{1:t-1}^{1,2})$ and $c_{t^+} = c_{t+1} := (z_{1:t},m_{1:t}^{1,2})$ be realizations of $C_t$, $C_{t^+}$ and $C_{t+1}$ respectively. Let $\gamma_{1:t},\lambda_{1:t}$ be the realizations of the coordinator's prescriptions $\Gamma_{1:t},\Lambda_{1:t}$ up to time $t$. Let us assume that the realizations $c_{t+1},\gamma_{1:t},\lambda_{1:t}$ have non-zero probability. Let $\pi_t^i$, $\pi_{t^+}^i$ and $\pi_{t+1}^i$ be the corresponding realizations of the coordinator's beliefs $\Pi_t^i$, $\Pi_{t^+}^i$, and $\Pi_{t+1}^i$ respectively. These beliefs are given by
\begin{align}
\pi_{t}^i(x_t^i)&=\prob(X_t^i=x_t^i|C_t=(z_{1:t-1},m^{1,2}_{1:t-1}),\Gamma_{1:t-1}=\gamma_{1:t-1},\Lambda_{1:t-1}=\lambda_{1:t-1})\label{updatelem3:1}\\
    \pi_{t^+}^i(x_{t^+}^i)&=\prob(X_{t}^i=x_{t}^i|C_{t^+}=(z_{1:t},m^{1,2}_{1:t}),\Gamma_{1:t}=\gamma_{1:t},\Lambda_{1:t-1}=\lambda_{1:t-1}).\label{updatelem3:2}\\
    \pi_{t+1}^i(x_{t+1}^i)&=\prob(X_{t+1}^i=x_{t+1}^i|C_{t+1}=(z_{1:t},m^{1,2}_{1:t}),\Gamma_{1:t}=\gamma_{1:t},\Lambda_{1:t}=\lambda_{1:t}).
\end{align}

There are two possible cases: (i) $Z_t = (\tilde{x}_t^{1},\tilde{x}_t^{2})$ for some $(\tilde{x}_t^{1},\tilde{x}_t^{2}) \in \mathcal{X}^1\times \mathcal{X}^2$ or (ii) $Z_t = \phi$. Let us analyze these two cases separately.

\textbf{Case I:} When $Z_t = (\tilde{x}_t^{1},\tilde{x}_t^{2})$ for some $(\tilde{x}_t^{1},\tilde{x}_t^{2}) \in \mathcal{X}^1\times \mathcal{X}^2$, at least one of the agents must have decided to communicate at time $t$. As described in \eqref{zt}, the variable $Z_t = X_t$ when communication occurs. Thus, we have
\begin{align}
    \pi_{t^+}^i(x_t^i)&=\mathds{1}_{(\tilde{x}^{i}_{t}=x^i_t)}. \label{updatelem3:3}
\end{align}

\textbf{Case II:} When ${Z}_t=\phi$, ${M}^{1,2}_t=(0,0)$ (see \eqref{zt}). Using the Bayes' rule, we have
\begin{align}
    \pi_{t^+}^i(x_t^i)&=P(X_t^i=x_{t}^i|z_{1:t},m^{1,2}_{1:t},\gamma_{1:t},\lambda_{1:t-1})\notag\\
    &=\frac{P(X_t^i=x_{t}^i,Z_t=\phi,M^{1,2}_t=(0,0)|z_{1:t-1},m^{1,2}_{1:t-1},\gamma_{1:t},\lambda_{1:t-1})}{P(Z_t=\phi,M^{1,2}_t=(0,0)|z_{1:t-1},m^{1,2}_{1:t-1},\gamma_{1:t},\lambda_{1:t-1})}\notag\\
    &=\frac{P(Z_t=\phi|x_{t}^i,c_t,\gamma_{1:t},\lambda_{1:t-1},M^{1,2}_t=(0,0))P(M^{1,2}_t=(0,0)|x_t^i,c_t,\gamma_{1:t},\lambda_{1:t-1}) P(X_t^i=x_{t}^i|c_t,\gamma_{1:t},\lambda_{1:t-1})   }{\sum_{\hat{x}_t}P(Z_t=\phi|\hat{x}_{t}^i,c_t,\gamma_{1:t},\lambda_{1:t-1},M^{1,2}_t=(0,0))P(M^{1,2}_t=(0,0)|\hat{x}_t^i,c_t,\gamma_{1:t},\lambda_{1:t-1}) P(X_t^i=\hat{x}_{t}^i|c_t,\gamma_{1:t},\lambda_{1:t-1})   }\notag\\
    &\stackrel{a}{=}\frac{P(M^{1,2}_t=(0,0)|x_t^i,c_t,\gamma_{1:t},\lambda_{1:t-1}) P(X_t^i=x_{t}^i|c_t,\gamma_{1:t},\lambda_{1:t-1})  }{\sum_{\hat{x}_t^i} P(M^{1,2}_t=(0,0)|\hat{x}_t^i,c_t,\gamma_{1:t},\lambda_{1:t-1})P(X_t^i=\hat{x}_{t}^i|c_t,\gamma_{1:t},\lambda_{1:t-1}) }\notag\\
    &\stackrel{b}{=}\frac{P(M^{1,2}_t=(0,0)|x_t^i,c_t,\gamma_{1:t},\lambda_{1:t-1}) P(X_t^i=x_{t}^i|c_t,\gamma_{1:t-1},\lambda_{1:t-1})  }{\sum_{\hat{x}_t^i} P(M^{1,2}_t=(0,0)|\hat{x}_t^i,c_t,\gamma_{1:t},\lambda_{1:t-1})P(X_t^i=\hat{x}_{t}^i|c_t,\gamma_{1:t-1},\lambda_{1:t-1}) }\notag\\
    %&\stackrel{c}{=}\frac{\mathds{1}_{(\gamma^1_t(x^1_t)=0)}\mathds{1}_{(\gamma^2_t(x^2_t)=0)}P(X_t=x_{t}|c_t,\gamma_{1:t-1},\lambda_{1:t-1}) }{\sum_{\hat{x}_t}\mathds{1}_{(\gamma^1_t(\hat{x}^1_t)=0)}\mathds{1}_{(\gamma^2_t(\hat{x}^2_t)=0)}P(X_t=\hat{x}_{t}|c_t,\gamma_{1:t-1},\lambda_{1:t-1})}\notag\\ 
    &\stackrel{c}{=}\frac{\mathds{1}_{(\gamma^i_t(x^i_t)=0)}\pi_t^i(x_t^i)}{\sum_{\hat{x}_t^i}\mathds{1}_{(\gamma^i_t(\hat{x}^i_t)=0)}\pi_t^i(\hat{x}_t^i)}.\label{updatelem3:4}
\end{align}
In the display above, equation $(a)$ follows from the fact that $Z_t = \phi$ if and only if $M^{1,2}_t = (0,0)$ (see \eqref{zt}). In equation $(b)$, we drop $\gamma_t$ from the term $ P(X_t^i=x_{t}^i|c_t,\gamma_{1:t},\lambda_{1:t-1})$ because $\gamma_t$ is a function of the rest of them terms in the conditioning given the coordinator's strategy. Due to Lemma \ref{LEM:INDEPEN}, $X_t^1$ and $X_t^2$ are independent conditioned\footnote{Given the coordinator's stratey, conditioning on $C_t$ and conditioning on $C_t,\Gamma_{1:t},\Lambda_{1:t-1}$ are the same because the prescriptions are functions of the common information.} on $C_t = c_t$. This conditional independence property and the fact that $M_t^i = \Gamma_t^i(X_t^i)$, (see \eqref{commact1}) leads to equation $(c)$.
Hence, we can update the coordinator's beliefs $\pi_{t^+}^i$ ($i=1,2$) using $\pi_t^i,\gamma_t^i$ and $z_t$ as:
\begin{equation}
  \pi^i_{t^+}(x_t^i)=\begin{cases}
    \frac{\mathds{1}_{(\gamma^i_t(x^i_t)=0)}\pi_t^i(x_t^i)}{\sum_{\hat{x}^i_t}\mathds{1}_{(\gamma^i_t(\hat{x}^i_t)=0)}\pi_t^i(\hat{x}_t^i)}, & \text{if $Z_t=\phi$}.\\
    \mathds{1}_{(x^i_t=\Tilde{x}^i_t)}, & \text{if $Z_t =  (\tilde{x}_t^{1},\tilde{x}_t^{2})$}.
  \end{cases}\label{updatelem3:11}
\end{equation}
We denote the update rule described above with $\eta^i_t$, i.e.
\begin{align}
\pi_{t^+}^i = \eta_t^i(\pi_t^i,\gamma_t^i,z_t).
\end{align}

Further, using the law of total probability, we have

\begin{align}
    &\pi_{t+1}^i(x_{t+1}^i)\\
    &=\prob(X_{t+1}^i =x_{t+1}^i|z_{1:t},m^{1,2}_{1:t},\gamma_{1:t},\lambda_{1:t})\\
    &= \sum_{x_t^i}\sum_{u_t^i}\prob(X_{t+1}^i = x_{t+1}^i|x_t^i,u_t^i,z_{1:t},m^{1,2}_{1:t},\gamma_{1:t},\lambda_{1:t})\prob(U_t^i = u_t^i|x_t^i,z_{1:t},m^{1,2}_{1:t},\gamma_{1:t},\lambda_{1:t})\prob(X_t^i  =x_{t}^i|z_{1:t},m^{1,2}_{1:t},\gamma_{1:t},\lambda_{1:t})\\
    &\stackrel{a}{=} \sum_{x_t^i}\sum_{u_t^i}\prob(X_{t+1}^i = x_{t+1}^i|x_t^i,u_t^i,z_{1:t},m^{1,2}_{1:t},\gamma_{1:t},\lambda_{1:t})\prob(U_t^i = u_t^i|x_t^i,z_{1:t},m^{1,2}_{1:t},\gamma_{1:t},\lambda_{1:t})\prob(X_t^i  =x_{t}^i|z_{1:t},m^{1,2}_{1:t},\gamma_{1:t},\lambda_{1:t-1})\\
    &\stackrel{b}{=} \sum_{x_t^i}\sum_{u_t^i}\prob(X_{t+1}^i = x_{t+1}^i|x_t^i,u_t^i,z_{1:t},m^{1,2}_{1:t},\gamma_{1:t},\lambda_{1:t})\mathds{1}_{(u_t^i=\lambda_t^i(x_t^i))}\pi_{t^+}^i(x_t^i)\\
    &\stackrel{c}{=} \sum_{x_t^i}\sum_{u_t^i}\prob(X_{t+1}^i = x_{t+1}^i|x_t^i,u_t^i)\mathds{1}_{(u_t^i=\lambda_t^i(x_t^i))}\pi_{t^+}^i(x_t^i).
\end{align}
In equation $(a)$ in the display above, we drop $\lambda_t$ from $\prob(X_t^i  =x_{t}^i|z_{1:t},m^{1,2}_{1:t},\gamma_{1:t},\lambda_{1:t})$ since $\lambda_t$ is a function of the rest of the terms in the conditioning given the coordinator's strategy. Equation $(b)$ follows from  \eqref{controlact1}. Equation $(c)$ follows from the system dynamics in \eqref{dyna}.
We denote the update rule described above with $\beta^i_t$, i.e.
\begin{align}
\pi_{t+1}^i = \beta_t^i(\pi_{t^+}^i,\lambda_t^i).
\end{align}

%erasure
%%%Erasure Proof%%%

\section{Erasure Model}\label{proof:Erasure model}
 Based on the erasure communication model described in Section 5, the variable $Z^{er}_t$ (information exchanged between agents in the erasure model) is defined as:
\begin{equation}
  Z^{er}_t=\begin{cases}
    X^{1,2}_t, ~ w.p. ~1-p_e & \text{if $M^{or}_t=1$}.\\
    \phi, ~~~~~ w.p. ~~~~p_e & \text{if $M^{or}_t=1$}.\\
    \phi, & \text{if $M^{or}_t=0$}.
  \end{cases}
\end{equation}
In the following subsections, we prove Lemma \ref{LEM:INDEPEN}, Proposition I and Lemma \ref{LEM:UPDATE} for the erasure communication model.
\subsection{Proof of Lemma \ref{LEM:INDEPEN}}
\label{proof:CI1}
We prove the lemma by induction. At $t=1$, before communication decisions are made, \eqref{eq:indepen} is trivially true since there is no common information at this point and the agents' initial states are independent. For \eqref{eq:indepen2} at $t=1$, recall that $c_{1^+} = (z^{er}_1, m_1)$. The left hand side of \eqref{eq:indepen2} for $t=1$ can be written as

\begin{align}
& \prob({x}_1,{u}_{1}|z^{er}_{1}, {m}_{1})=\frac{\prob({x}_1,{u}_{1},z^{er}_{1},{m}_{1})}{\prob(z^{er}_{1},{m}_{1})}  \notag\\
&=\frac{\prob({u}_{1}|{x}_1,z^{er}_{1},{m}_{1})\prob(z^{er}_{1}|{x}_1,{m}_{1})\prob(m_{1}|x_1)\prob(x_1)}{\prob(z^{er}_{1},m_{1})} \notag \\
 &\frac{\mathds{1}_{(u^1_1=g^1_1(x^1_1,z^{er}_{1},m_1))}\mathds{1}_{(u^2_1=g^2_1(x^2_1,z^{er}_{1},m_1))}\prob(z^{er}_{1}|x_1,m_{1})\mathds{1}_{(m^2_1=f^2_1(x^2_1))}\mathds{1}_{(m^1_1=f^1_1(x^1_1))}\prob(x^1_1)\prob(x^2_1)}{\prob(z^{er}_{1},m_{1})}. \label{ern11:base}
\end{align}

The first term (corresponding to $i=1$) on the right hand side of \eqref{eq:indepen2} for $t=1$ can be written as 
\begin{align}
     &\prob(x^1_1,u^1_{1}|z^{er}_{1},m_{1})=\frac{\prob(x^1_1,u^1_{1},z^{er}_{1},m_{1})}{\prob(z^{er}_{1},m_{1})}\notag\\
     &=\frac{\prob(u^1_{1}|x^1_1,z^{er}_{1},m_{1})\prob(z^{er}_{1}|x^1_1,m_{1})\prob(m_{1}|x^1_1)\prob(x^1_1)}{\prob(z^{er}_{1},m_{1})}\notag\\
    &=\frac{\mathds{1}_{(u^1_1=g^1_1(x^1_1,z^{er}_1,m_1))}\prob(z^{er}_{1}|x^1_1,m_{1})\mathds{1}_{(m^1_1=f^1_1(x^1_1))}\prob(m^2_{1}|x^1_1)\prob(x^1_1)}{\prob(z^{er}_{1},m_{1})}. \label{eafter11:comm}
\end{align}
Similarly, the second term (corresponding to $i=2$) on the right hand side of \eqref{eq:indepen2}  for $t=1$ can be written as
\begin{equation}\label{eafter22:comm}
    \frac{\mathds{1}_{(u^2_1=g^2_1(x^2_1,z^{er}_1,m_1))}\prob(z^{er}_{1}|x^2_1,m_{1})\mathds{1}_{(m^2_1=f^2_1(x^2_1))}\prob(m^1_{1}|x^2_1)\prob(x^2_1)}{\prob(z^{er}_{1},m_{1})}.
\end{equation}
Comparing \eqref{ern11:base}, \eqref{eafter11:comm} and \eqref{eafter22:comm}, it is clear that  we just need to prove that
\begin{align}
\frac{\prob(z^{er}_{1}|x_1,m_{1})}{\prob(z^{er}_{1},m_{1})}  = \frac{\prob(z^{er}_{1}|x^1_1,m_{1})\prob(m^2_{1}|x^1_1)}{\prob(z^{er}_{1},m_{1})} \times \frac{\prob(z^{er}_{1}|x^2_1,m_{1})\prob(m^1_{1}|x^2_1)}{\prob(z^{er}_{1},m_{1})} \label{eeq:new_11}
\end{align}
in order to establish \eqref{eq:indepen2} for $t=1$. We consider three cases:

\textbf{Case I: } $Z^{er}_1=\phi$ and $M_1 = (0,0)$. In this case, the left hand side of \eqref{eeq:new_11} can be written as 
\begin{align}
\frac{\prob(Z^{er}_{1}= \phi|X_1=x_1,M_{1}=(0,0))}{\prob(Z^{er}_{1}=\phi,M_{1} = (0,0))} = \frac{1}{\prob(M^1_{1} = 0)\prob(M^2_1=0)} \label{eeq:new66}
\end{align}
Similarly, the right hand side of  \eqref{eeq:new_11} can be written as 
\begin{align}
 &\frac{P(M^2_1=0|X^1_1=x^1_1)}{\prob(M^1_{1} = 0)\prob(M^2_1=0)} \times \frac{P(M^1_1=0|X^2_1=x^2_1)}{\prob(M^1_{1} = 0)\prob(M^2_1=0)} \notag \\
 &= \frac{P(M^2_1=0)}{\prob(M^1_{1} = 0)\prob(M^2_1=0)} \times \frac{P(M^1_1=0)}{\prob(M^1_{1} = 0)\prob(M^2_1=0)} \notag \\
 &= \frac{1}{\prob(M^1_{1} = 0)\prob(M^2_1=0)} \label{eeq:new55}
\end{align}
where we used the fact that $M^i_1$ and $X^j_1$ are independent for $i \neq j$. Comparing \eqref{eeq:new66} and \eqref{eeq:new55} establishes \eqref{eeq:new_11} for this case.

\textbf{Case II: } $Z^{er}_1=(\tilde{x}^1_1, \tilde{x}^2_1) $ and $M_1 = (m^1_1,m^2_1) \neq (0,0)$. \footnote{This case occurs only if $p_e<1$.}In this case, the left hand side of \eqref{eeq:new_11} can be written as 
\begin{align}
&\frac{\prob(Z^{er}_{1}= (\tilde{x}^1_1, \tilde{x}^2_1)|X_1=x_1,M_{1}=(m^1_1,m^2_1))}{\prob(Z^{er}_{1}= (\tilde{x}^1_1, \tilde{x}^2_1),M_{1}=(m^1_1,m^2_1))} = \frac{(1-p_e)\mathds{1}_{\{x_1=(\tilde{x}^1_1, \tilde{x}^2_1)\}}}{(1-p_e)\prob(X_{1}= (\tilde{x}^1_1, \tilde{x}^2_1),M_{1}=(m^1_1,m^2_1))} \notag \\
&= \frac{\mathds{1}_{\{x_1=(\tilde{x}^1_1, \tilde{x}^2_1)\}}}{\prob(X_{1}= (\tilde{x}^1_1, \tilde{x}^2_1))\mathds{1}_{(m^1_1=f^1_1(\tilde{x}^1_1))}\mathds{1}_{(m^2_1=f^2_1(\tilde{x}^2_1))}} \label{eeq:new22}
\end{align}
Consider the first term on the right hand side of \eqref{eeq:new_11}. It can be written as
\begin{align}
&\frac{\prob(Z^{er}_{1} = (\tilde{x}^1_1, \tilde{x}^2_1) |X^1_1 = x^1_1, M_1 =m_{1})\prob(M^2_1 = m^2_{1}|X^1_1=x^1_1)}{\prob(Z^{er}_{1}= (\tilde{x}^1_1, \tilde{x}^2_1),M_{1}=(m^1_1,m^1_2))} \notag \\
& =  \frac{(1-p_e)\mathds{1}_{\{x^1_1=\tilde{x}^1_1\}}\prob(X^2_1 = \tilde{x}^2_1|M^2_1 = m^2_1)\prob(M^2_1 = m^2_{1}) }{(1-p_e)\prob(X_{1}= (\tilde{x}^1_1, \tilde{x}^2_1))\mathds{1}_{(m^1_1=f^1_1(\tilde{x}^1_1))}\mathds{1}_{(m^2_1=f^2_1(\tilde{x}^2_1))}} \notag \\
&= \frac{\mathds{1}_{\{x^1_1=\tilde{x}^1_1\}}\prob(X^2_1 = \tilde{x}^2_1, M^2_1 = m^2_1) }{\prob(X_{1}= (\tilde{x}^1_1, \tilde{x}^2_1))\mathds{1}_{(m^1_1=f^1_1(\tilde{x}^1_1))}\mathds{1}_{(m^2_1=f^2_1(\tilde{x}^2_1))}} \notag \\
&= \frac{\mathds{1}_{\{x^1_1=\tilde{x}^1_1\}}\prob(X^2_1 = \tilde{x}^2_1)\mathds{1}_{(m^2_1=f^2_1(\tilde{x}^2_1))} }{\prob(X_{1}= (\tilde{x}^1_1, \tilde{x}^2_1))\mathds{1}_{(m^1_1=f^1_1(\tilde{x}^1_1))}\mathds{1}_{(m^2_1=f^2_1(\tilde{x}^2_1))}} \notag \\
&= \frac{\mathds{1}_{\{x^1_1=\tilde{x}^1_1\}} }{\prob(X^1_{1}= \tilde{x}^1_1)\mathds{1}_{(m^1_1=f^1_1(\tilde{x}^1_1))}} \label{eeq:new33}
\end{align} 
where we used the fact that $(X^i_1,M^i_1)$ are independent of  $X^j_1$  for $i \neq j$. Similarly, the second term on the right hand side of  \eqref{eeq:new_11} can be written as 
\begin{align}
&\frac{\prob(Z^{er}_{1} = (\tilde{x}^1_1, \tilde{x}^2_1) |X^2_1 = x^1_1, M_1 =m_{1})\prob(M^1_1 = m^1_{1}|X^2_1=x^2_1)}{\prob(Z^{er}_{1}= (\tilde{x}^1_1, \tilde{x}^2_1),M_{1}=(m^1_1,m^1_2))} \notag \\
&= \frac{\mathds{1}_{\{x^2_1=\tilde{x}^2_1\}} }{\prob(X^2_{1}= \tilde{x}^2_1)\mathds{1}_{(m^2_1=f^2_1(\tilde{x}^2_1))}} \label{eeq:new44}
\end{align} 
Comparing \eqref{eeq:new22}, \eqref{eeq:new33} and \eqref{eeq:new44} establishes \eqref{eeq:new_11} for this case.

\textbf{Case III: } $Z^{er}_1=\phi $ and $M_1 = (m^1_1,m^2_1) \neq (0,0)$. \footnote{This case occurs only if $p_e>0$.}In this case, the left hand side of \eqref{eeq:new_11} can be written as
\begin{align}
&\frac{\prob(Z^{er}_{1}= \phi|X_1=x_1,M_{1}=(m^1_1,m^2_1))}{\prob(Z^{er}_{1}= \phi,M_{1}=(m^1_1,m^2_1))} = \frac{p_e}{p_e\prob(M_{1}=(m^1_1,m^2_1))} \notag \\
&= \frac{1}{\prob(M^1_1=m^1_1)\prob(M^2_1=m^2_1)} \label{eeq:n2}
\end{align}
Consider the first term on the right hand side of \eqref{eeq:new_11}. It can be written as
\begin{align}
&\frac{\prob(Z^{er}_{1} = \phi |X^1_1 = x^1_1, M_1 =m_{1})\prob(M^2_1 = m^2_{1}|X^1_1=x^1_1)}{\prob(Z^{er}_{1}= \phi,M_{1}=(m^1_1,m^2_1))} 
=  \frac{p_e\prob(M^2_1 = m^2_{1}) }{p_e\prob(M_{1}=(m^1_1,m^2_1)))} \notag \\
&= \frac{1}{\prob(M^1_{1}=m^1_1)} \label{eeq:n3}
\end{align} 
where we used the fact that $(X^i_1,M^i_1)$ are independent of  $X^j_1$  for $i \neq j$. Similarly, the second term on the right hand side of  \eqref{eeq:new_11} can be written as 
\begin{align}
&\frac{\prob(Z^{er}_{1} = \phi |X^2_1 = x^1_1, M_1 =m_{1})\prob(M^1_1 = m^1_{1}|X^2_1=x^2_1)}{\prob(Z^{er}_{1}=\phi ,M_{1}=(m^1_1,m^2_1))} 
= \frac{1}{\prob(M^2_{1}=m^2_1)} \label{eeq:n4}
\end{align} 
Comparing \eqref{eeq:n2}, \eqref{eeq:n3} and \eqref{eeq:n4} establishes \eqref{eeq:new_11} for this case.

\textbf{Induction step:} Assuming that Lemma \ref{LEM:INDEPEN} holds for  $t$, we need to show that it holds for $t+1$. At time $t+1$, before 
communication decisions are made, left hand side of equation \eqref{eq:indepen} can be written as
\begin{align}
& \prob(x_{1:t+1},u_{1:t}|z^{er}_{1:t},m_{1:t})=\prob(x^{1,2}_{t+1}|x_{1:t},u_{1:t},z^{er}_{1:t},m_{1:t})\prob(x_{1:t},u_{1:t}|z^{er}_{1:t},m_{1:t})  \notag\\
&=\prob(x^{2}_{t+1}|x^{1}_{t+1},x_{1:t},u_{1:t},z^{er}_{1:t},m_{1:t})\prob(x^1_{t+1}|x_{1:t},u_{1:t},z^{er}_{1:t},m_{1:t})P(x_{1:t},u_{1:t}|z^{er}_{1:t},m_{1:t}) \notag \\
 &=\prob(x^{2}_{t+1}|x^2_{t},u^2_{t})\prob(x^{1}_{t+1}|x^1_{t},u^1_{t})\prob(x^1_{1:t},u^1_{1:t}|z^{er}_{1:t},m_{1:t})\prob(x^2_{1:t},u^2_{1:t}|z^{er}_{1:t},m_{1:t}). \label{elnn1:inductive}
\end{align}

The first term (corresponding to $i=1$) on the right hand side of \eqref{eq:indepen} for $t+1$ can be written as 
\begin{align}
     &\prob(x^1_{1:t+1},u^1_{1:t}|z^{er}_{1:t},m_{1:t})=\prob(x^{1}_{t+1}|x^1_{1:t},u^1_{1:t},z^{er}_{1:t},m_{1:t})P(x^1_{1:t},u^1_{1:t}|z^{er}_{1:t},m_{1:t})\notag\\
     &=\prob(x^{1}_{t+1}|x^1_{t},u^1_{t})P(x^1_{1:t},u^1_{1:t}|z^{er}_{1:t},m_{1:t}). \label{ernn1:inductive}
\end{align}
Similarly for $i=2$,
\begin{equation}\label{ernn2:inductive}
    P(x^2_{1:t+1},u^2_{1:t}|z^{er}_{1:t},m_{1:t})=\prob(x^{2}_{t+1}|x^2_{t},u^2_{t})P(x^2_{1:t},u^2_{1:t}|z^{er}_{1:t},m_{1:t}).
\end{equation}
Comparing \eqref{elnn1:inductive}, \eqref{ernn1:inductive} and \eqref{ernn2:inductive} establishes \eqref{eq:indepen}   for time $t+1$.

For \eqref{eq:indepen2} at $t+1$, recall that $c_{{t+1}^+} = (z^{er}_{1:t+1},m_{1:t+1})$. The left hand side of \eqref{eq:indepen2} for $t+1$ can be written as 
\begin{align}
&\prob(x_{1:t+1},u_{1:t+1}|z^{er}_{1:t+1},m_{1:t+1})=
\frac{\prob(x_{1:t+1},u_{1:t+1},z^{er}_{t+1},m_{t+1}| z^{er}_{1:t},m_{1:t})}{\prob(z^{er}_{t+1},m_{t+1} | z^{er}_{1:t},m_{1:t})} \notag\\
&=\prob(u_{t+1}|x_{1:t+1},u_{1:t},z^{er}_{1:t+1},m_{1:t+1})\times\notag\\
&\frac{\prob(z^{er}_{t+1}|x_{1:t+1},u_{1:t},z^{er}_{1:t},m_{1:t+1})\prob(m_{t+1}|x_{1:t+1},u_{1:t},z^{er}_{1:t},m_{1:t})\prob(x_{1:t+1},u_{1:t}|z^{er}_{1:t},m_{1:t})}{\prob(z^{er}_{t+1},m_{t+1} | z^{er}_{1:t},m_{1:t})} \notag \\
 &=\mathds{1}_{(u^2_{t+1}=g^2_{t+1}(x^2_{1:t+1},u^2_{1:t},z^{er}_{1:t+1},m_{1:t+1}))}\mathds{1}_{(u^1_{t+1}=g^1_{t+1}(x^1_{1:t+1},u^1_{1:t},z^{er}_{1:t+1},m_{1:t+1}))}\prob(z^{er}_{t+1}|x_{1:t+1},u_{1:t},z^{er}_{1:t},m_{1:t+1})\times \notag\\
&\frac{\mathds{1}_{(m^2_{t+1}=f^2_{t+1}(x^2_{1:t+1},u^2_{1:t},z^{er}_{1:t},m_{1:t}))}\mathds{1}_{(m^1_{t+1}=f^1_{t+1}(x^1_{1:t+1},u^1_{1:t},z^{er}_{1:t},m_{1:t}))}\prob(x^1_{1:t+1},u^1_{1:t}|z^{er}_{1:t},m_{1:t})\prob(x^2_{1:t+1},u^2_{1:t}|z^{er}_{1:t},m_{1:t})}{\prob(z^{er}_{t+1},m_{t+1} | z^{er}_{1:t},m_{1:t})}. \label{erhss2:inductive}
\end{align}

The first term (corresponding to $i=1$) on the right hand side of \eqref{eq:indepen2} for $t+1$ can be written as 
\begin{align}
     &\prob(x^1_{1:t+1},u^1_{1:t+1}|z^{er}_{1:t+1},m_{1:t+1})=\frac{\prob(x^1_{1:t+1},u^1_{1:t+1},z^{er}_{t+1},m_{t+1}| z^{er}_{1:t},m_{1:t})}{\prob(z^{er}_{t+1},m_{t+1} | z^{er}_{1:t},m_{1:t})}\notag\\
     &=\prob(u^1_{t+1}|x^1_{1:t+1},u^1_{1:t},z^{er}_{1:t+1},m_{1:t+1}) \times \notag\\
     &\frac{\prob(z^{er}_{t+1}|x^1_{1:t+1},u^1_{1:t},z^{er}_{1:t},m_{1:t+1})\prob(m_{t+1}|x^1_{1:t+1},u^1_{1:t},z^{er}_{1:t},m_{1:t})\prob(x^1_{1:t+1},u^1_{1:t}|z^{er}_{1:t},m_{1:t})}{\prob(z^{er}_{t+1},m_{t+1} | z^{er}_{1:t},m_{1:t})}\notag\\
    &=\mathds{1}_{(u^1_{t+1}=g^1_{t+1}(x^1_{1:t+1},u^1_{1:t},z^{er}_{1:t+1},m_{1:t+1}))}\prob(z^{er}_{t+1}|x^1_{1:t+1},u^1_{1:t},z^{er}_{1:t},m_{1:t+1})\times \notag\\
&\frac{\mathds{1}_{(m^1_{t+1}=f^1_{t+1}(x^1_{1:t+1},u^1_{1:t},z^{er}_{1:t},m_{1:t})}\prob(m^2_{t+1}|x^1_{1:t+1},u^1_{1:t},z^{er}_{1:t},m_{1:t})\prob(x^1_{1:t+1},u^1_{1:t}|z^{er}_{1:t},m_{1:t})}{\prob(z^{er}_{t+1},m_{t+1} | z^{er}_{1:t},m_{1:t})}. \label{erhss2:comm}
\end{align}
Similarly, the second  term (corresponding to $i=2$) on the right hand side of \eqref{eq:indepen2} for $t+1$ can be written as 
\begin{align}\label{erhss2:t+1}
&\mathds{1}_{(u^2_{t+1}=g^2_{t+1}(x^2_{1:t+1},u^2_{1:t},z^{er}_{1:t+1},m_{1:t+1}))}\prob(z^{er}_{t+1}|x^2_{1:t+1},u^2_{1:t},z^{er}_{1:t},m_{1:t+1})\times\notag\\
&\frac{\prob(m^1_{t+1}|x^2_{1:t+1},u^2_{1:t},z^{er}_{1:t},m_{1:t})\mathds{1}_{(m^2_{t+1}=f^2_{t+1}(x^2_{1:t+1},u^2_{1:t},z^{er}_{1:t},,m_{1:t})}\prob(x^2_{1:t+1},u^2_{1:t}|z^{er}_{1:t},m_{1:t})}{\prob(z^{er}_{t+1},m_{t+1} | z^{er}_{1:t},m_{1:t})}.
\end{align}

Comparing \eqref{erhss2:inductive}, \eqref{erhss2:comm} and \eqref{erhss2:t+1}, it is clear that  we just need to prove that
\begin{align}
\frac{\prob(z^{er}_{t+1}|x_{1:t+1},u_{1:t},z^{er}_{1:t},m_{1:t+1})}{\prob(z^{er}_{t+1},m_{t+1} | z^{er}_{1:t},m_{1:t})} &= \frac{\prob(z^{er}_{t+1}|x^1_{1:t+1},u^1_{1:t},z^{er}_{1:t},m_{1:t+1})\prob(m^2_{t+1}|x^1_{1:t+1},u^1_{1:t},z_{1:t},m_{1:t})}{\prob(z^{er}_{t+1},m_{t+1} | z^{er}_{1:t},m_{1:t})} \times \notag\\ &\frac{\prob(z^{er}_{t+1}|x^2_{1:t+1},u^2_{1:t},z^{er}_{1:t},m_{1:t+1})\prob(m^1_{t+1}|x^2_{1:t+1},u^2_{1:t},z^{er}_{1:t},m_{1:t})}{\prob(z^{er}_{t+1},m_{t+1} | z^{er}_{1:t},m_{1:t})}. \label{eeq:neww_11}
\end{align}
in order to establish \eqref{eq:indepen2} for $t+1$. We consider three cases:

\textbf{Case I: } $Z^{er}_{t+1}=\phi$ and $M_{t+1} = (0,0)$. In this case, the left hand side of \eqref{eeq:neww_11} can be written as

\begin{align}
\frac{\prob(Z^{er}_{t+1}= \phi|x_{1:t+1},u_{1:t},z^{er}_{1:t},m_{1:t},M_{t+1}=(0,0))}{\prob(Z^{er}_{t+1}=\phi,M_{t+1} = (0,0)|z^{er}_{1:t},m_{1:t})} = \frac{1}{\prob(M^1_{t+1} = 0|z^{er}_{1:t},m_{1:t})\prob(M^2_{t+1}=0|z^{er}_{1:t},m_{1:t})} \label{eeq:ne6}
\end{align}
Similarly, the right hand side of  \eqref{eeq:neww_11} can be written as 
\begin{align}
 &\frac{\prob(M^2_{t+1}=0|x^1_{1:t+1},u^1_{1:t},z^{er}_{1:t},m_{1:t})}{\prob(M^1_{t+1} = 0|z^{er}_{1:t},m_{1:t})\prob(M^2_{t+1}=0|z^{er}_{1:t},m_{1:t})} \times \frac{\prob(M^1_{t+1}=0|x^2_{1:t+1},u^2_{1:t},z^{er}_{1:t},m_{1:t})}{\prob(M^1_{t+1} = 0|z^{er}_{1:t},m_{1:t})\prob(M^2_{t+1}=0|z^{er}_{1:t},m_{1:t})} \notag \\
 &= \frac{\prob(M^2_{t+1}=0|z^{er}_{1:t},m_{1:t})}{\prob(M^1_{t+1} = 0|z^{er}_{1:t},m_{1:t})\prob(M^2_{t+1}=0|z^{er}_{1:t},m_{1:t})} \times \frac{\prob(M^1_{t+1} = 0|z^{er}_{1:t},m_{1:t})}{\prob(M^1_{t+1} = 0|z^{er}_{1:t},m_{1:t})\prob(M^2_{t+1}=0|z^{er}_{1:t},m_{1:t})} \notag \\
 &=\frac{1}{\prob(M^1_{t+1} = 0|z^{er}_{1:t},m_{1:t})\prob(M^2_{t+1}=0|z^{er}_{1:t},m_{1:t})} \label{eeq:ne5}
\end{align}
where we used the fact that $M^i_{t+1}$ and $X^j_{t+1}$ are independent for $i \neq j$. Comparing \eqref{eeq:ne6} and \eqref{eeq:ne5} establishes \eqref{eeq:neww_11} for this case.

\textbf{Case II: } $Z^{er}_{t+1}=(\tilde{x}^1_{t+1}, \tilde{x}^2_{t+1}) $ and $M_{t+1} = (m^1_{t+1},m^2_{t+1}) \neq (0,0)$. \footnote{This case occurs only if $p_e<1$.}In this case, the left hand side of \eqref{eeq:neww_11} can be written as 
\begin{align}
&\frac{\prob(Z^{er}_{t+1}= (\tilde{x}^1_{t+1}, \tilde{x}^2_{t+1})|x_{1:t+1},u_{1:t},z^{er}_{1:t},m_{1:t},M_{t+1}=(m^1_{t+1},m^2_{t+1}))}{\prob(Z^{er}_{t+1}=(\tilde{x}^1_{t+1}, \tilde{x}^2_{t+1}),M_{t+1} = (m^1_{t+1},m^2_{t+1})|z^{er}_{1:t},m_{1:t})} \notag\\
&= \frac{(1-p_e)\mathds{1}_{\{x_{t+1}=(\tilde{x}^1_{t+1}, \tilde{x}^2_{t+1})\}}}{(1-p_e)\prob(X_{t+1}=(\tilde{x}^1_{t+1}, \tilde{x}^2_{t+1}),M_{t+1} = (m^1_{t+1},m^2_{t+1})|z^{er}_{1:t},m_{1:t})} \notag \\
&= \frac{\mathds{1}_{\{x_{t+1}=(\tilde{x}^1_{t+1}, \tilde{x}^2_{t+1})\}}}{\prob(X^1_{t+1}=\tilde{x}^1_{t+1}, M^1_{t+1}= m^1_{t+1}|z^{er}_{1:t},m_{1:t}) \prob(X^2_{t+1}=\tilde{x}^2_{t+1}, M^2_{t+1} =m^2_{t+1}|z^{er}_{1:t},m_{1:t})} \label{eeq:neww22}
\end{align}
where we used  \eqref{eq:indepen}   for time $t+1$.
Consider the first term on the right hand side of \eqref{eeq:neww_11}. It can be written as
\begin{align}
&\frac{\prob(Z^{er}_{t+1}=(\tilde{x}^1_{t+1}, \tilde{x}^2_{t+1})|x^1_{1:t+1},u^1_{1:t},z^{er}_{1:t},m_{1:t},M_{t+1}=(m^1_{t+1},m^2_{t+1}))\prob(M^2_{t+1}=m^2_{t+1}|x^1_{1:t+1},u^1_{1:t},z^{er}_{1:t},m_{1:t})}{\prob(Z^{er}_{t+1}=(\tilde{x}^1_{t+1}, \tilde{x}^2_{t+1}),M_{t+1} = (m^1_{t+1},m^2_{t+1})|z^{er}_{1:t},m_{1:t})} \notag\\
& =  \frac{(1-p_e)\mathds{1}_{\{x^1_{t+1}=\tilde{x}^1_{t+1}\}}\prob({X}^2_{t+1}=\tilde{x}^2_{t+1}|z^{er}_{1:t},m_{1:t},m^2_{t+1})\prob(m^2_{t+1}|z^{er}_{1:t},m_{1:t}) }{(1-p_e)\prob(X^1_{t+1}=\tilde{x}^1_{t+1}, M^1_{t+1}= m^1_{t+1}|z^{er}_{1:t},m_{1:t}) \prob(X^2_{t+1}=\tilde{x}^2_{t+1}, M^2_{t+1} =m^2_{t+1}|z^{er}_{1:t},m_{1:t})} \notag \\
& =  \frac{\mathds{1}_{\{x^1_{t+1}=\tilde{x}^1_{t+1}\}}\prob({X}^2_{t+1}=\tilde{x}^2_{t+1},M^2_{t+1}=m^2_{t+1}|z^{er}_{1:t},m_{1:t}) }{\prob(X^1_{t+1}=\tilde{x}^1_{t+1}, M^1_{t+1}= m^1_{t+1}|z^{er}_{1:t},m_{1:t}) \prob(X^2_{t+1}=\tilde{x}^2_{t+1}, M^2_{t+1} =m^2_{t+1}|z^{er}_{1:t},m_{1:t})} \notag\\
&=\frac{\mathds{1}_{\{x^1_{t+1}=\tilde{x}^1_{t+1}\}} }{\prob(X^1_{t+1}=\tilde{x}^1_{t+1}, M^1_{t+1}= m^1_{t+1}|z^{er}_{1:t},m_{1:t})}
\label{eeq:neww33}
\end{align} 
where we used the fact that $(X^i_{t+1},M^i_{t+1})$ are independent of  $X^j_{t+1}$  for $i \neq j$. Similarly, the second term on the right hand side of  \eqref{eeq:neww_11} can be written as 
\begin{align}
&\frac{\mathds{1}_{\{x^2_{t+1}=\tilde{x}^2_{t+1}\}} }{\prob(X^2_{t+1}=\tilde{x}^2_{t+1}, M^2_{t+1}= m^2_{t+1}|z^{er}_{1:t},m_{1:t})}\label{eeq:neww44}
\end{align} 
Comparing \eqref{eeq:neww22}, \eqref{eeq:neww33} and \eqref{eeq:neww44} establishes \eqref{eeq:neww_11} for this case.

\textbf{Case III: } $Z^{er}_{t+1}=\phi $ and $M_{t+1} = (m^1_{t+1},m^2_{t+1}) \neq (0,0)$. \footnote{This case occurs only if $p_e>0$.}In this case, the left hand side of \eqref{eeq:neww_11} can be written as

\begin{align}
&\frac{\prob(Z^{er}_{t+1}= \phi|x_{1:t+1},u_{1:t},z^{er}_{1:t},m_{1:t},M_{t+1}=(m^1_{t+1},m^2_{t+1}))}{\prob(Z^{er}_{t+1}=\phi,M_{t+1} = (m^1_{t+1},m^2_{t+1})|z^{er}_{1:t},m_{1:t})} = \frac{p_e}{p_e\prob(M^1_{t+1} = m^1_{t+1}|z^{er}_{1:t},m_{1:t})\prob(M^2_{t+1}=m^2_{t+1}|z^{er}_{1:t},m_{1:t})} \notag\\
&=\frac{1}{\prob(M^1_{t+1} = m^1_{t+1}|z^{er}_{1:t},m_{1:t})\prob(M^2_{t+1}=m^2_{t+1}|z^{er}_{1:t},m_{1:t})}\label{eeq:neww66}
\end{align}
Similarly, the right hand side of  \eqref{eeq:neww_11} can be written as 
\begin{align}
 &\frac{p_e\prob(M^2_{t+1}=m^2_{t+1}|x^1_{1:t+1},u^1_{1:t},z^{er}_{1:t},m_{1:t})}{p_e\prob(M^1_{t+1} = m^1_{t+1}|z^{er}_{1:t},m_{1:t})\prob(M^2_{t+1}=m^2_{t+1}|z^{er}_{1:t},m_{1:t})} \times \frac{p_e\prob(M^1_{t+1}=m^1_{t+1}|x^2_{1:t+1},u^2_{1:t},z^{er}_{1:t},m_{1:t})}{p_e\prob(M^1_{t+1} = m^1_{t+1}|z^{er}_{1:t},m_{1:t})\prob(M^2_{t+1}=m^2_{t+1}|z^{er}_{1:t},m_{1:t})} \notag \\
 &= \frac{\prob(M^2_{t+1}=m^2_{t+1}|z^{er}_{1:t},m_{1:t})}{\prob(M^1_{t+1} = m^1_{t+1}|z^{er}_{1:t},m_{1:t})\prob(M^2_{t+1}=m^2_{t+1}|z^{er}_{1:t},m_{1:t})} \times \frac{\prob(M^1_{t+1} = m^1_{t+1}|z^{er}_{1:t},m_{1:t})}{\prob(M^1_{t+1} = m^1_{t+1}|z^{er}_{1:t},m_{1:t})\prob(M^2_{t+1}=m^2_{t+1}|z^{er}_{1:t},m_{1:t})} \notag \\
 &=\frac{1}{\prob(M^1_{t+1} = m^1_{t+1}|z^{er}_{1:t},m_{1:t})\prob(M^2_{t+1}=m^2_{t+1}|z^{er}_{1:t},m_{1:t})} \label{eeq:neww55}
\end{align}
where we used the fact that $M^i_{t+1}$ and $X^j_{t+1}$ are independent for $i \neq j$. Comparing \eqref{eeq:neww66} and \eqref{eeq:neww55} establishes \eqref{eeq:neww_11} for this case.

%%Appendix B_erasure

\subsection{Proof of Proposition 1} \label{proof:propp1}

 We will prove the result for agent $i$. Throughout this proof, we fix agent $-i's$ communication and control strategies to be  $f^{-i}$,$g^{-i}$ (where f$^{-i}$,$g^{-i}$ are arbitrarily chosen). Define $R^i_t=(X^i_{t},Z^{er}_{1:t-1},M^{1,2}_{1:t-1})$ and $R^i_{t^+}=(X^i_{t},Z^{er}_{1:t},M^{1,2}_{1:t})$. Our proof will rely on  the following two facts:
 
  \textbf{Fact 1}: $\{R^i_1, R^i_{1^+}, R^i_2, R^i_{^2+}, ....R^i_T, R^i_{T^+}\}$ is a controlled Markov process for agent $i$. 
  More precisely, for any strategy choice $f^i, g^i$ of agent $i$, 
%Write (A) and (B),
\begin{align}
&\prob(R^i_{t^+}=\tilde{r}^i_{t^+}|R^i_{1:t}=r^i_{1:t},M^i_{1:t}=m^i_{1:t})=\prob(R^i_{t^+}=\tilde{r}^i_{t^+}|R^i_{t}=r^i_{t},M^i_{t}=m^i_{t}) \label{eq:mdppA} \\
& \prob(R^{i}_{t+1}=\tilde{r}^{i}_{t+1}|R^i_{1:{t^+}}=r^i_{1:{t^+}},U^i_{1:t}=u^i_{1:t})=\prob(R^{i}_{t+1}=\tilde{r}^{i}_{t+1}|R^i_{{t^+}}=r^i_{{t^+}},U^i_{t}=u^i_{t}) \label{eq:mdppB}
\end{align}
where the probabilities on the right hand side of  \eqref{eq:mdppA} and \eqref{eq:mdppB} do not depend on $f^i,g^i$.

  \textbf{Fact 2}: The costs at time $t$ satisfy
%%Write LHS of C = k^i_t(R^i_t, M^i_t)
%%Write LHS of D = k^i_{t^+}(R^i_t^+, U^i_t),
\begin{align}
& \ee[\rho\mathds{1}_{(M^{or}_t=1)}|R^i_{1:t}=r^i_{1:t},M^i_{1:t}=m^i_{1:t}]= \kappa^i_t(r^i_t, m^i_t) \label{eq:mdppC}\\
& \ee[c_t({X}_t,{U}_t)|R^i_{1:{t^+}}=r^i_{1:{t^+}},U^i_{1:t}=u^i_{1:t}] =\kappa^i_{t^+}(r^i_{t^+}, u^i_t) \label{eq:mdppD}
 \end{align}
 %(R^i_{t^+}, U^i_t)
 %=\k^i_{t^+}(R^i_t^+, U^i_t) 
  where the functions $\kappa^i_t, \kappa^i_{t^+}$ in \eqref{eq:mdppC} and \eqref{eq:mdppD}  do not depend on $f^i, g^i$.

Suppose that Facts 1 and 2 are true. Then, the strategy optimization problem for agent $i$ can be viewed as a MDP over $2T$ time steps (i.e. time steps $1, 1^+, 2, 2^+, \ldots, T, T^+$) with $R^i_t$ and $M^i_t$ as the state and action at time $t$; and $R^i_{t^+}$ and $U^i_t$ as the state and action for time $t^+$. Note that at time $t$, agent $i$ observes $R^i_t$, selects $M^i_t$ and the ``state'' transitions to $R^i_{t^+}$ according to Markovian dynamics (see \eqref{eq:mdppA}). Similarly, at time $t^+$, agent $i$ observes $R^i_{t^+}$, selects $U^i_t$ and the ``state'' transitions to $R^i_{t+1}$ according to Markovian dynamics (see \eqref{eq:mdppB}). Further, from agent $i$'s perspective, the cost at time $t$ depends on the current state and action (i.e. $R^i_t$ and $M^i_t$, see \eqref{eq:mdppC}) and the cost at time $t^+$ depends on the state and action at $t^+$ (i.e. $R^i_{t^+}$ and $U^i_t$, see \eqref{eq:mdppD}).  It then follows from standard MDP results that agent $i$'s strategy should be of the form: 
\begin{equation} 
   M^i_t=  \bar{f}^i_t(R^i_{t}) =  \bar{f}^i_t(X^i_{t},Z^{er}_{1:t-1},M^{1,2}_{1:t-1}),
\end{equation}

\begin{equation}  
   U^i_t=  \bar{g}^i_t(R_{t^+}) = \bar{g}^i_t(X^i_{t}, Z^{er}_{1:t},M^{1,2}_{1:t}),
\end{equation}
which establishes the result of the proposition (recall that $C_t = (Z^{er}_{1:t-1},M^{1,2}_{1:t-1})$ and $C_{t^+}=(Z^{er}_{1:t},M^{1,2}_{1:t}$).

We now prove Facts 1 and 2. 

(i) Let $\tilde{r}^i_{t^+} = ({x}^i_t,{z}^{er}_{1:t},{m}_{1:t})$ and $r^i_{1:t} = (x^i_{1:t},z^{er}_{1:t-1},m_{1:t-1})$. Then, the left hand side of \eqref{eq:mdppA} can be written as 
\begin{align}
 &\prob(R^i_{t^+}=({x}^i_t,{z}^{er}_{1:t},{m}_{1:t})|R^i_{1:t}=(x^i_{1:t},z^{er}_{1:t-1},m_{1:t-1}),M^i_{1:t}=m^i_{1:t})\notag\\
    &=\prob({Z}^{er}_t={z}^{er}_t|x^i_{1:t},z^{er}_{1:t-1},m_{1:t})\prob({M}^{-i}_t={m}^{-i}_t|x^i_{1:t},z^{er}_{1:t-1},m_{1:t-1},m^i_{t}) \notag\\
    &= \prob({Z}^{er}_t={z}^{er}_t|x^i_{1:t},z^{er}_{1:t-1},m_{1:t})\prob({M}^{-i}_t={m}^{-i}_t|z^{er}_{1:t-1},m_{1:t-1},m^i_{t})  \label{eq:markk4}
\end{align}
where \eqref{eq:markk4} follows from the conditional independence property of Lemma \ref{LEM:INDEPEN}.
 We can further simplify the first term in  \eqref{eq:markk4} for different cases as follows: 
 
\textbf{Case I:} when $Z^{er}_t=\tilde{x}^{1,2}_t$ and ${M}_t=(m^1_t,m^2_t) \neq (0,0)$
\begin{align}
    \prob(Z^{er}_t=\tilde{x}^{1,2}_t|x^1_{1:t},z^{er}_{1:t-1},m_{1:t})=(1-p_e)\mathds{1}_{(\tilde{x}^i_{t}={x}^i_{t})}\prob(\tilde{x}^{-i}_t|z^{er}_{1:t-1},m_{1:t}) \label{ztt:A1}
\end{align}
\textbf{Case II:} when $Z^{er}_t=\phi$ and ${M}_t=(0,0)$
\begin{align}
    \prob(Z^{er}_t=\phi|x^1_{1:t},z_{1:t-1},m_{1:t})=1 \label{ztt:A2}
\end{align}
\textbf{Case III:} when $Z^{er}_t=\phi$ and ${M}_t=(m^1_t,m^2_t) \neq (0,0)$
\begin{align}
    \prob(Z^{er}_t=\phi|x^1_{1:t},z_{1:t-1},m_{1:t})=p_e \label{ztt:A3}
\end{align}
We note that in all  cases above $x^i_{1:t-1}$ does not affect the probability.  Further, the probabilities in the three cases do not depend on agent $i$'s strategy.

Repeating the above steps for the right hand side of \eqref{eq:mdppA} establishes that the three sides of  \eqref{eq:mdppA} are equal.

(ii) \eqref{eq:mdppB} is a direct consequence of the Markovian state dynamics of agent $i$.

(iii) In \eqref{eq:mdppC}, it is straightforward to see that if $m^i_t =1$, then the left hand side is simply $\rho$. If, on the other hand, $m^i_t =0$, then the left hand side of \eqref{eq:mdppC} can be written as 
\begin{align}
&\rho\prob(M^{-i}_t =1|r^i_{1:t},m^i_{1:t}) = \rho\prob({M}^{-i}_t=1|x^i_{1:t},z^{er}_{1:t-1},m_{1:t-1},m^i_{t}) \notag \\
&=\rho\prob({M}^{-i}_t=1|z^{er}_{1:t-1},m_{1:t-1},m^i_{t}) \label{eq:markkA1}
\end{align}
where \eqref{eq:markkA1} follows from the conditional independence property of Lemma \ref{LEM:INDEPEN}. The right hand side of \eqref{eq:markkA1} is a function only of $r^i_t$ and $m^i_t$ and does not depend on agent $i$'s strategy. This completes the proof of \eqref{eq:mdppC}. 

(iv) To prove \eqref{eq:mdppD}, it suffices to show that 
\begin{align}
 \prob(x^{-i}_t,u^{-i}_t|(x^i_{1:t},z^{er}_{1:t},m_{1:t}),u^i_{1:t})=  \prob(x^{-i}_t,u^{-i}_t|(x^i_{t},z^{er}_{1:t},m_{1:t}),u^i_{t}) \label{eq:markkA2}
\end{align}
\eqref{eq:markkA2} follows from the conditional independence property of Lemma \ref{LEM:INDEPEN}.

%proof of belief update

\subsection{Proof of Lemma \ref{LEM:UPDATE}}\label{proof:Coordd}
Recall that at the beginning of time $t$, the common information is given by $C_{t}:=(Z^{er}_{1:t-1},M^{1,2}_{1:t-1})$ (see \eqref{commoninfo}).
At the end of time $t$, i.e. after the communication decisions are made at time $t$, the common information is given by $C_{t^+}:=(Z^{er}_{1:t},M^{1,2}_{1:t})$ (see \eqref{commoninfoplus}). Let $c_{t} := (z^{er}_{1:t-1},m_{1:t-1}^{1,2})$ and $c_{t^+} = c_{t+1} := (z^{er}_{1:t},m_{1:t}^{1,2})$ be realizations of $C_t$, $C_{t^+}$ and $C_{t+1}$ respectively. Let $\gamma_{1:t},\lambda_{1:t}$ be the realizations of the coordinator's prescriptions $\Gamma_{1:t},\Lambda_{1:t}$ up to time $t$. Let us assume that the realizations $c_{t+1},\gamma_{1:t},\lambda_{1:t}$ have non-zero probability. Let $\pi_t^i$, $\pi_{t^+}^i$ and $\pi_{t+1}^i$ be the corresponding realizations of the coordinator's beliefs $\Pi_t^i$, $\Pi_{t^+}^i$, and $\Pi_{t+1}^i$ respectively. These beliefs are given by
\begin{align}
\pi_{t}^i(x_t^i)&=\prob(X_t^i=x_t^i|C_t=(z^{er}_{1:t-1},m^{1,2}_{1:t-1}),\Gamma_{1:t-1}=\gamma_{1:t-1},\Lambda_{1:t-1}=\lambda_{1:t-1})\label{updatelemm3:1}\\
    \pi_{t^+}^i(x_{t^+}^i)&=\prob(X_{t}^i=x_{t}^i|C_{t^+}=(z^{er}_{1:t},m^{1,2}_{1:t}),\Gamma_{1:t}=\gamma_{1:t},\Lambda_{1:t-1}=\lambda_{1:t-1}).\label{updatelemm3:2}\\
    \pi_{t+1}^i(x_{t+1}^i)&=\prob(X_{t+1}^i=x_{t+1}^i|C_{t+1}=(z^{er}_{1:t},m^{1,2}_{1:t}),\Gamma_{1:t}=\gamma_{1:t},\Lambda_{1:t}=\lambda_{1:t}).
\end{align}

There are three possible cases: (i) $Z^{er}_t = (\tilde{x}_t^{1},\tilde{x}_t^{2})$ and $M_t=(m^1_t,m^2_t)\neq (0,0)$, (ii)$Z^{er}_t = \phi$ and $M_t=(m^1_t,m^2_t)\neq (0,0)$ and (iii) $Z^{er}_t = \phi$ and $M_t=(0,0)$. Let us analyze these three cases separately.

\textbf{Case I:} When $Z^{er}_t = (\tilde{x}_t^{1},\tilde{x}_t^{2})$ for some $(\tilde{x}_t^{1},\tilde{x}_t^{2}) \in \mathcal{X}^1\times \mathcal{X}^2$ and $M_t=(m^1_t,m^2_t)\neq (0,0)$, at least one of the agents must have decided to communicate at time $t$. Thus, we have
\begin{align}
    \pi_{t^+}^i(x_t^i)&=P(X_t^i=x_{t}^i|z^{er}_{1:t},m^{1,2}_{1:t},\gamma_{1:t},\lambda_{1:t-1})\notag\\
    &=\frac{P(X_t^i=x_{t}^i,Z^{er}_t=(\tilde{x}_t^{1},\tilde{x}_t^{2}),M^{1,2}_t=(m^1_t,m^2_t)|z^{er}_{1:t-1},m^{1,2}_{1:t-1},\gamma_{1:t},\lambda_{1:t-1})}{P(Z^{er}_t=(\tilde{x}_t^{1},\tilde{x}_t^{2}),M^{1,2}_t=(m^1_t,m^2_t)|z^{er}_{1:t-1},m^{1,2}_{1:t-1},\gamma_{1:t},\lambda_{1:t-1})}\notag\\
    &=\frac{(1-p_e)P(M^{1,2}_t=(m^1_t,m^2_t)|X_t=(\tilde{x}_t^{1},\tilde{x}_t^{2}),x_t^i,c_t,\gamma_{1:t},\lambda_{1:t-1}) P(X_t^i=x_{t}^i|X_t=(\tilde{x}_t^{1},\tilde{x}_t^{2}),c_t,\gamma_{1:t},\lambda_{1:t-1})   }{(1-p_e)P(M^{1,2}_t=(m^1_t,m^2_t)|X_t=(\tilde{x}_t^{1},\tilde{x}_t^{2}),x_t^i,c_t,\gamma_{1:t},\lambda_{1:t-1})  }\notag\\
    &=\mathds{1}_{(x^i_t=\Tilde{x}^i_t)}. \label{updatelemm3:3}
\end{align}

\textbf{Case II:} When ${Z}^{er}_t=\phi$, $M_t=(m^1_t,m^2_t)\neq (0,0)$ (see \eqref{erasure:model}). Using the Bayes' rule, we have
\begin{align}
    \pi_{t^+}^i(x_t^i)&=P(X_t^i=x_{t}^i|z^{er}_{1:t},m^{1,2}_{1:t},\gamma_{1:t},\lambda_{1:t-1})\notag\\
    &=\frac{P(X_t^i=x_{t}^i,Z^{er}_t=\phi,M^{1,2}_t=(m^1_t,m^2_t)|z^{er}_{1:t-1},m^{1,2}_{1:t-1},\gamma_{1:t},\lambda_{1:t-1})}{P(Z^{er}_t=\phi,M^{1,2}_t=(m^1_t,m^2_t)|z^{er}_{1:t-1},m^{1,2}_{1:t-1},\gamma_{1:t},\lambda_{1:t-1})}\notag\\
    &=\frac{P(Z^{er}_t=\phi|x_{t}^i,c_t,\gamma_{1:t},\lambda_{1:t-1},M^{1,2}_t=(m^1_t,m^2_t))P(M^{1,2}_t=(m^1_t,m^2_t)|x_t^i,c_t,\gamma_{1:t},\lambda_{1:t-1}) P(X_t^i=x_{t}^i|c_t,\gamma_{1:t},\lambda_{1:t-1})   }{\sum_{\hat{x}_t}P(Z^{er}_t=\phi|\hat{x}_{t}^i,c_t,\gamma_{1:t},\lambda_{1:t-1},M^{1,2}_t=((m^1_t,m^2_t))P(M^{1,2}_t=(m^1_t,m^2_t)|\hat{x}_t^i,c_t,\gamma_{1:t},\lambda_{1:t-1}) P(X_t^i=\hat{x}_{t}^i|c_t,\gamma_{1:t},\lambda_{1:t-1})   }\notag\\
    &\stackrel{a}{=}\frac{p_e P(M^{1,2}_t=(m^1_t,m^2_t)|x_t^i,c_t,\gamma_{1:t},\lambda_{1:t-1}) P(X_t^i=x_{t}^i|c_t,\gamma_{1:t},\lambda_{1:t-1})  }{\sum_{\hat{x}_t^i} p_e P(M^{1,2}_t=(m^1_t,m^2_t)|\hat{x}_t^i,c_t,\gamma_{1:t},\lambda_{1:t-1})P(X_t^i=\hat{x}_{t}^i|c_t,\gamma_{1:t},\lambda_{1:t-1}) }\notag\\
    &\stackrel{b}{=}\frac{P(M^{1,2}_t=(m^1_t,m^2_t)|x_t^i,c_t,\gamma_{1:t},\lambda_{1:t-1}) P(X_t^i=x_{t}^i|c_t,\gamma_{1:t-1},\lambda_{1:t-1})  }{\sum_{\hat{x}_t^i} P(M^{1,2}_t=(m^1_t,m^2_t)|\hat{x}_t^i,c_t,\gamma_{1:t},\lambda_{1:t-1})P(X_t^i=\hat{x}_{t}^i|c_t,\gamma_{1:t-1},\lambda_{1:t-1}) }\notag\\
    &\stackrel{c}{=}\frac{\mathds{1}_{(\gamma^i_t(x^i_t)=m^i_t)}\pi_t^i(x_t^i)}{\sum_{\hat{x}_t^i}\mathds{1}_{(\gamma^i_t(\hat{x}^i_t)=m^i_t)}\pi_t^i(\hat{x}_t^i)}.\label{updatelemm3:4}
\end{align}
In the display above, equation $(a)$ follows from the definition of Erasure model (see \eqref{erasure:model}). In equation $(b)$, we drop $\gamma_t$ from the term $ P(X_t^i=x_{t}^i|c_t,\gamma_{1:t},\lambda_{1:t-1})$ because $\gamma_t$ is a function of the rest of the terms in the conditioning given the coordinator's strategy. Due to Lemma \ref{LEM:INDEPEN}, $X_t^1$ and $X_t^2$ are independent conditioned on $C_t = c_t$. This conditional independence property and the fact that $M_t^i = \Gamma_t^i(X_t^i)$, (see \eqref{commact1}) leads to equation $(c)$.

\textbf{Case III:} When ${Z}^{er}_t=\phi$, ${M}^{1,2}_t=(0,0)$ (see \eqref{erasure:model}). Using the Bayes' rule, we have
\begin{align}
    \pi_{t^+}^i(x_t^i)&=P(X_t^i=x_{t}^i|z^{er}_{1:t},m^{1,2}_{1:t},\gamma_{1:t},\lambda_{1:t-1})\notag\\
    &=\frac{P(X_t^i=x_{t}^i,Z^{er}_t=\phi,M^{1,2}_t=(0,0)|z^{er}_{1:t-1},m^{1,2}_{1:t-1},\gamma_{1:t},\lambda_{1:t-1})}{P(Z^{er}_t=\phi,M^{1,2}_t=(0,0)|z^{er}_{1:t-1},m^{1,2}_{1:t-1},\gamma_{1:t},\lambda_{1:t-1})}\notag\\
    &=\frac{P(Z^{er}_t=\phi|x_{t}^i,c_t,\gamma_{1:t},\lambda_{1:t-1},M^{1,2}_t=(0,0))P(M^{1,2}_t=(0,0)|x_t^i,c_t,\gamma_{1:t},\lambda_{1:t-1}) P(X_t^i=x_{t}^i|c_t,\gamma_{1:t},\lambda_{1:t-1})   }{\sum_{\hat{x}_t}P(Z^{er}_t=\phi|\hat{x}_{t}^i,c_t,\gamma_{1:t},\lambda_{1:t-1},M^{1,2}_t=(0,0))P(M^{1,2}_t=(0,0)|\hat{x}_t^i,c_t,\gamma_{1:t},\lambda_{1:t-1}) P(X_t^i=\hat{x}_{t}^i|c_t,\gamma_{1:t},\lambda_{1:t-1})   }\notag\\
    &\stackrel{a}{=}\frac{P(M^{1,2}_t=(0,0)|x_t^i,c_t,\gamma_{1:t},\lambda_{1:t-1}) P(X_t^i=x_{t}^i|c_t,\gamma_{1:t},\lambda_{1:t-1})  }{\sum_{\hat{x}_t^i} P(M^{1,2}_t=(0,0)|\hat{x}_t^i,c_t,\gamma_{1:t},\lambda_{1:t-1})P(X_t^i=\hat{x}_{t}^i|c_t,\gamma_{1:t},\lambda_{1:t-1}) }\notag\\
    &\stackrel{b}{=}\frac{P(M^{1,2}_t=(0,0)|x_t^i,c_t,\gamma_{1:t},\lambda_{1:t-1}) P(X_t^i=x_{t}^i|c_t,\gamma_{1:t-1},\lambda_{1:t-1})  }{\sum_{\hat{x}_t^i} P(M^{1,2}_t=(0,0)|\hat{x}_t^i,c_t,\gamma_{1:t},\lambda_{1:t-1})P(X_t^i=\hat{x}_{t}^i|c_t,\gamma_{1:t-1},\lambda_{1:t-1}) }\notag\\
    &\stackrel{c}{=}\frac{\mathds{1}_{(\gamma^i_t(x^i_t)=0)}\pi_t^i(x_t^i)}{\sum_{\hat{x}_t^i}\mathds{1}_{(\gamma^i_t(\hat{x}^i_t)=0)}\pi_t^i(\hat{x}_t^i)}.\label{updatelemm3:4}
\end{align}
In the display above, equation $(a)$ follows from the fact that $Z^{er}_t = \phi$  if $M^{1,2}_t = (0,0)$ (see \eqref{erasure:model}). In equation $(b)$, we drop $\gamma_t$ from the term $ P(X_t^i=x_{t}^i|c_t,\gamma_{1:t},\lambda_{1:t-1})$ because $\gamma_t$ is a function of the rest of them terms in the conditioning given the coordinator's strategy. Due to Lemma \ref{LEM:INDEPEN}, $X_t^1$ and $X_t^2$ are independent conditioned\footnote{Given the coordinator's strategy, conditioning on $C_t$ and conditioning on $C_t,\Gamma_{1:t},\Lambda_{1:t-1}$ are the same because the prescriptions are functions of the common information.} on $C_t = c_t$. This conditional independence property and the fact that $M_t^i = \Gamma_t^i(X_t^i)$, (see \eqref{commact1}) leads to equation $(c)$.
Hence, we can update the coordinator's beliefs $\pi_{t^+}^i$ ($i=1,2$) using $\pi_t^i,\gamma_t^i$ and $z^{er}_t$ as:
\begin{equation}
  \pi^i_{t^+}(x_t^i)=\begin{cases}
    \frac{\mathds{1}_{(\gamma^i_t(x^i_t)=0)}\pi_t^i(x_t^i)}{\sum_{\hat{x}^i_t}\mathds{1}_{(\gamma^i_t(\hat{x}^i_t)=0)}\pi_t^i(\hat{x}_t^i)}, & \text{if $Z^{er}_t=\phi$ and $M_t=(0,0)$}.\\
    \frac{\mathds{1}_{(\gamma^i_t(x^i_t)=m^i_t)}\pi_t^i(x_t^i)}{\sum_{\hat{x}^i_t}\mathds{1}_{(\gamma^i_t(\hat{x}^i_t)=m^i_t)}\pi_t^i(\hat{x}_t^i)}, & \text{if $Z^{er}_t=\phi$ and $M_t=(m^1_t,m^2_t)$}.\\
    \mathds{1}_{(x^i_t=\Tilde{x}^i_t)}, & \text{if $Z^{er}_t =  (\tilde{x}_t^{1},\tilde{x}_t^{2})$ and $M_t=(m^1_t,m^2_t)$}.
  \end{cases}\label{updatelem3:11}
\end{equation}
We denote the update rule described above with $\eta^{er_i}_t$, i.e.
\begin{align}
\pi_{t^+}^i = \eta^{er_i}_t(\pi_t^i,\gamma_t^i,z^{er}_t).
\end{align}

Further, using the law of total probability, we have

\begin{align}
    &\pi_{t+1}^i(x_{t+1}^i)\\
    &=\prob(X_{t+1}^i =x_{t+1}^i|z^{er}_{1:t},m^{1,2}_{1:t},\gamma_{1:t},\lambda_{1:t})\\
    &= \sum_{x_t^i}\sum_{u_t^i}\prob(X_{t+1}^i = x_{t+1}^i|x_t^i,u_t^i,z^{er}_{1:t},m^{1,2}_{1:t},\gamma_{1:t},\lambda_{1:t})\prob(U_t^i = u_t^i|x_t^i,z^{er}_{1:t},m^{1,2}_{1:t},\gamma_{1:t},\lambda_{1:t})\prob(X_t^i  =x_{t}^i|z^{er}_{1:t},m^{1,2}_{1:t},\gamma_{1:t},\lambda_{1:t})\\
    &\stackrel{a}{=} \sum_{x_t^i}\sum_{u_t^i}\prob(X_{t+1}^i = x_{t+1}^i|x_t^i,u_t^i,z^{er}_{1:t},m^{1,2}_{1:t},\gamma_{1:t},\lambda_{1:t})\prob(U_t^i = u_t^i|x_t^i,z^{er}_{1:t},m^{1,2}_{1:t},\gamma_{1:t},\lambda_{1:t})\prob(X_t^i  =x_{t}^i|z^{er}_{1:t},m^{1,2}_{1:t},\gamma_{1:t},\lambda_{1:t-1})\\
    &\stackrel{b}{=} \sum_{x_t^i}\sum_{u_t^i}\prob(X_{t+1}^i = x_{t+1}^i|x_t^i,u_t^i,z^{er}_{1:t},m^{1,2}_{1:t},\gamma_{1:t},\lambda_{1:t})\mathds{1}_{(u_t^i=\lambda_t^i(x_t^i))}\pi_{t^+}^i(x_t^i)\\
    &\stackrel{c}{=} \sum_{x_t^i}\sum_{u_t^i}\prob(X_{t+1}^i = x_{t+1}^i|x_t^i,u_t^i)\mathds{1}_{(u_t^i=\lambda_t^i(x_t^i))}\pi_{t^+}^i(x_t^i).
\end{align}
In equation $(a)$ in the display above, we drop $\lambda_t$ from $\prob(X_t^i  =x_{t}^i|z^{er}_{1:t},m^{1,2}_{1:t},\gamma_{1:t},\lambda_{1:t})$ since $\lambda_t$ is a function of the rest of the terms in the conditioning given the coordinator's strategy. Equation $(b)$ follows from  \eqref{controlact1}. Equation $(c)$ follows from the system dynamics in \eqref{dyna}.
We denote the update rule described above with $\beta^{er_i}_t$, i.e.
\begin{align}
\pi_{t+1}^i = \beta^{er_i}_t(\pi_{t^+}^i,\lambda_t^i).
\end{align}

\end{document}